\documentclass[12pt,xcolor=svgnames]{amsart}

\usepackage{amsmath, amsthm, amssymb, amsfonts, verbatim, dsfont}

\usepackage{xspace,mathtools,bm}

\usepackage[svgnames]{xcolor}

\usepackage[pagebackref=true, colorlinks=true, citecolor=blue]{hyperref}

\allowdisplaybreaks

\DeclareRobustCommand{\rchi}{{\mathpalette\irchi\relax}}
\newcommand{\irchi}[2]{\raisebox{\depth}{$#1\chi$}}

\newtheorem{theorem}{Theorem}[section]
\newtheorem{cor}[theorem]{Corollary}
\newtheorem{prop}[theorem]{Proposition}

\newtheorem{lemma}[theorem]{Lemma}

\theoremstyle{definition}

\newtheorem*{theorem*}{Theorem}

\theoremstyle{remark}
\newtheorem{remark}[theorem]{Remark}

\newcommand{\Obsolete}[1]{
    }

\usepackage[english]{babel}

\numberwithin{equation}{section}

\title
    [2D Euler in Logarithmically Refined Sobolev Spaces]
    {Ill-posedness of the 2D Euler Equations in a Logarithmically Refined Critical Sobolev Space}

\author{Elaine Cozzi}
\address{Oregon State University, Department of Mathematics}
\email{elaine.cozzi@oregonstate.edu}

\author{Nicholas Harrison}
\address{Oregon State University, Department of Mathematics}
\email{harrnich@oregonstate.edu}

\author{Zachary Radke}
\address{Oregon State University, Department of Mathematics}
\email{radkeza@oregonstate.edu}
\begin{document}

\begin{abstract}
    In \cite{bourgainli15}, Bourgain and Li establish strong ill-posedness of the 2D Euler equations for initial velocity in the critical Sobolev space $H^2(\mathbb{R}^2)$.  In this work, we extend the results of \cite{bourgainli15} by demonstrating strong ill-posedness in logarithmically regularized spaces which are strictly contained in $H^2(\mathbb{R}^2)$ and which contain $H^s(\mathbb{R}^2)$ for all $s>2$. These spaces are constructed via application of a fractional logarithmic derivative to the critical Sobolev norm. We show that if the power $\alpha$ of the logarithmic derivative satisfies $\alpha\leq 1/2$, then the 2D Euler equations are strongly ill-posed.
\end{abstract}
\maketitle
%%%%%%%%%%%%%%%%%%%%%%%%%%%%%%%%%
%%%%%%%%%%%%%%%%%%%%%%%%%%%%%%%%%
%%%%%%%%%%%%%%%%%%%%%%%%%%%%%%%%%
\section{Introduction}\label{Section Introduction}
%%%%%%%%%%%%%%%%%%%%%%%%%%%%%%%%%
%%%%%%%%%%%%%%%%%%%%%%%%%%%%%%%%%
%%%%%%%%%%%%%%%%%%%%%%%%%%%%%%%%%
The incompressible Euler equations on $\mathbb{R}^2$ in velocity form are given by
\begin{equation}\label{VelocityEquation}\tag{E}
\begin{cases}
    \partial_tu+u\cdot\nabla u = -\nabla p \quad &\text{on }\mathbb{R}^2\times(0,\infty),
    \\ \nabla\cdot u=0 \quad &\text{on }\mathbb{R}^2\times[0,\infty),
    \\ u|_{t=0}=u_0 \quad &\text{on }\mathbb{R}^2,    
\end{cases}    
\end{equation}
where $u:\mathbb{R}^2\times [0,\infty)\rightarrow \mathbb{R}^2$ denotes the velocity and $p:\mathbb{R}^2\times[0,\infty)\to\mathbb{R}$ denotes the internal scalar pressure. %Due to the incompressibility condition, $\nabla\cdot u=0$, the pressure can be determined via $p= (-\Delta)^{-1}\nabla\cdot\left(u\cdot\nabla u \right).$ 
The 2D incompressible Euler equations can be recast in terms of the \textit{vorticity}, defined by $$\omega:=\nabla^\perp \cdot u=\partial_{x_2}u_1-\partial_{x_1}u_2.$$ One can show using the incompressibility condition that $(\ref{VelocityEquation})$ is equivalent to the 2D vorticity equation, given by
\begin{equation}\label{VorticityEquation}\tag{V} 
\begin{cases}
    \partial_t\omega+u\cdot\nabla\omega = 0 \quad &\text{on }\mathbb{R}^2\times(0,\infty),
    \\ u = \nabla^\perp(-\Delta)^{-1}\omega  &\text{on }\mathbb{R}^2\times[0,\infty),
    \\ \omega|_{t=0}=\omega_0 &\text{on }\mathbb{R}^2.
\end{cases}
\end{equation}
Here, the velocity can be obtained from the vorticity via the relation $u=K*\omega$, where $$K(x)= \frac{1}{2\pi}\frac{x^\perp}{|x|^2}$$ is the Biot-Savart kernel. We refer to convolution with $K$ as the Biot-Savart operator. Throughout this work, we also let $X:\mathbb{R}^2\times [0,\infty)\to\mathbb{R}^2$ be the particle trajectory map defined by 
\begin{equation}\label{Flow map definition}
    \frac{d}{dt}X^t(x)=u(X^t(x),t), \ \ X^0(x)=x,\ \ \ \ (x,t)\in\mathbb{R}^2\times[0,\infty).
\end{equation}

The question of well-posedness of $(\ref{VelocityEquation})$ in different functional settings has a long history.  Wolibner \cite{wolibner33} establishes the global well-posedness of $(\ref{VelocityEquation})$ in $C^{k,r}(\mathbb{R}^2)$, the space of $k$-times continuously differentiable functions with the $k^{\text{th}}$-order derivatives $r$-H\"older continuous, when $k\geq 1$ and $r\in(0,1)$. Kato and Ponce \cite{katoponce86} prove local well-posedness of $(\ref{VelocityEquation})$ in the Sobolev spaces $W^{s,p}(\mathbb{R}^2)$, with $s>2/p+1$ and $p\in(1,\infty)$.  One can obtain global in time solutions in these Sobolev spaces via the so-called BKM criterion \cite{bealekatomajda84}, with at most double exponential growth in time of the $W^{s,p}$-norm. In the lower regularity setting, Yudovich \cite{yudovich63} establishes existence of a unique weak solution to $(\ref{VorticityEquation})$ in $L^\infty_{loc}([0,\infty);L^1\cap L^\infty(\mathbb{R}^2))$. See also \cite{crippaseisspirito17}, \cite{bohunbouchutcrippa16} for results on existence of Lagrangian and weak solutions with vorticity in $L^1(\mathbb{R}^2)$. For well-posedness results in the Besov and Triebel-Lizorkin spaces, $B^{s}_{p,q}(\mathbb{R}^2)$ and $F^{s}_{p,q}(\mathbb{R}^2)$, we refer the reader to \cite{bahourichemindanchin11, chae02, chae03, chae04, vishik98, hwangpak23}. Notably, for these spaces, the condition $s>2/p+1$ is sufficient for global well-posedness, while, when $s=2/p+1$, well-posedness can be established in certain cases when the corresponding space embeds into $L^\infty(\mathbb{R}^2)$. %Finally, most of the results above have three- and higher- dimensional analogues, but for only finite time.

In contrast to the well-posedness results mentioned above, several works in recent years have demonstrated strong ill-posedness of $(\ref{VelocityEquation})$ in various function spaces.  Here, strong ill-posedness in the Banach space $X$ refers to the existence of initial data $u_0\in X$ for which the unique (Yudovich) solution $u$ to $(\ref{VelocityEquation})$ does not belong to $L^\infty([0,t];X)$ for \textit{any} $t>0$. %Bahouri and Chemin \cite{bahourichemin94} show that when $u_0\in W^{s,p}(\mathbb{R}^2)$, $s< p/2 +1$, $p\in (1,\infty)$, there exists a continuous, increasing function $\sigma:[0,T]\to[0,\infty)$ such that $\sigma(0)=0$ and $u(t)\in W^{3-\sigma(t),p}(\mathbb{R}^2)$ for each $t>0$. Here $\sigma$ depends on the initial data. The authors show that $\sigma$ is sharp by constructing initial data such that the Sobolev regularity index decays like $\sigma(t)$. 
The seminal works of Bourgain and Li \cite{bourgainli15, bourgainli21} establish the strong ill-posedness of $(\ref{VelocityEquation})$ for $d=2$ and $d=3$ in $C^k(\mathbb{R}^d)$ with $k$ a positive integer and in $B^s_{p,q}(\mathbb{R}^d)$ with $s=d/p +1$, $p\in(1,\infty)$, and $q>1$.  The authors of \cite{elgindijeong17, jeongkim21} give simplified proofs of the results of \cite{bourgainli15, bourgainli21}, with \cite{elgindijeong17} providing a proof on the two-dimensional torus, while \cite{jeongkim21} also establishes ill-posedness in the critical Lorentz spaces. A similar construction to that used in \cite{bourgainli15} is applied in \cite{kwon21}, showing strong ill-posedness of the logarithmically modified Euler equations, where the vorticity is transported by a flow logarithmically smoother than the velocity given by the Biot-Savart law.    %Finally, in the continuous setting, the author of \cite{khalil25} considers the propagation of the modulus of continuity of the vorticity. It is shown that if $\omega_0$ is bounded and if its modulus of continuity $m_{\omega_0}(r):=\sup_{|x-y|\leq r}|\omega_0(x)-\omega_0(y)|$ vanishes fast enough as $r\to0$, then $m_{\omega(t)}$ has the same rate of vanishing. Notably, this includes such moduli as $m(r)=\log^{-\gamma}(r)$ and $\log(|\log r|)\log^{-\gamma}(r)$ for $\gamma\in(0,1)$, for which there is no a priori control of the supremum of the velocity gradient. On the other hand, they construct solutions which satisfy $m_{\omega(t)}(r)\geq \frac{1}{2}\log\left(|\log r|\right)m_{\omega_0}(r)$, thus showing some moduli are not preserved. 

In this work, we demonstrate ill-posedness of $(\ref{VelocityEquation})$ in logarithmically refined critical Sobolev spaces. Before stating our main result, we briefly discuss a few relevant function spaces (see Section \ref{Section Preliminaries and Notation} for further discussion of these spaces).  We first define a logarithmic derivative using the Fourier transform: we fix $\alpha\in \mathbb{R}$ and set
\begin{equation}\label{Logarithmic derivative}
    \log^\alpha(D+1)f := \int_{\mathbb{R}^2}\log^\alpha(|\xi|+1)\widehat{f}(\xi)e^{ix\cdot\xi}\ d\xi,
\end{equation} for a locally integrable function $f:\mathbb{R}^2\to\mathbb{R}$, where $\widehat{f}:\mathbb{R}^2\to\mathbb{R}$ denotes the Fourier transform of $f$. We can now define the logarithmic Sobolev seminorm 
\begin{equation}\label{Logarithmic Sobolev seminorm}
        |f|_{H^{s,\alpha}}^2:=\int_{\mathbb{R}^2} |\xi|^{2s}\log^{2\alpha}(|\xi|+1)|\widehat{f}(\xi)|^2d\xi,
\end{equation}
where $s\geq 0.$ Then the logarithmic Sobolev space $H^{s,\alpha}(\mathbb{R}^2)$ is given by $$H^{s,\alpha}(\mathbb{R}^2):=\{f\in L^1_{loc}(\mathbb{R}^2): \|f\|_{H^{s,\alpha}}^2:=\|f\|_{L^2}^2 +|f|_{H^{s,\alpha}}^2<\infty\}. $$ Of course, when $\alpha=0$, $H^{s,\alpha}(\mathbb{R}^2)$ reduces to the classical Sobolev space $H^{s}(\mathbb{R}^2)$.  Moreover, for any $\alpha>0$ and $\epsilon>0$, we have the (strict) embeddings $$H^{s+\varepsilon,0}(\mathbb{R}^2)\hookrightarrow H^{s,\alpha}(\mathbb{R}^2)\hookrightarrow H^{s,0}(\mathbb{R}^2).$$ Thus, when $\alpha>0$, the spaces $H^{2,\alpha}(\mathbb{R}^2)$ lie strictly between the supercritical scale, at which $(\ref{VelocityEquation})$ is well-posed, and the critical scale, at which $(\ref{VelocityEquation})$ is strongly ill-posed. Our main result states that if the extra regularity imposed by the $\alpha$-logarithmic derivative is to a low enough degree, then $(\ref{VelocityEquation})$ is strongly ill-posed in $H^{2,\alpha}(\mathbb{R}^2)$.
\begin{theorem}\label{main}
    Let $\alpha\in(0,\frac{1}{2}]$ and $\varepsilon>0$. Then there exists $u_0\in H^{2,\alpha}(\mathbb{R}^2)\cap C^\infty(\mathbb{R}^2)$ such that $\|u_0\|_{H^{2,\alpha}}<\varepsilon$ and the unique smooth solution $u$ to $(\ref{VelocityEquation})$ satisfies $$\limsup_{t\to 0^+}\|u(t)\|_{H^{2,\alpha}}=\infty.$$ 
\end{theorem}
We note that in the recent work \cite{harrisonradke25}, the second and third authors prove that $(\ref{VelocityEquation})$ is globally well-posed in $H^{2,\alpha}(\mathbb{R}^2)$ for $\alpha>\frac{1}{2}$;  thus, the degree of logarithmic regularity in Theorem \ref{main} is sharp.  

We briefly outline our strategy for proving Theorem \ref{main}, % strong ill-posedness in $H^{2,\alpha}(\mathbb{R}^2)$ for $\alpha\leq 1/2$
 which closely follows the method of large Lagrangian deformation used in the works \cite{bourgainli15, bourgainli21, kwon21}. First, in Section \ref{section Local Norm Inflation}, we construct initial vorticity which is small in the $H^{1,\alpha}$ norm, but which produces a large Lagrangian deformation in short time. We then perturb this initial data by a highly oscillatory function which is also small in $H^{1,\alpha}$-norm. We show that this perturbation, combined with the stretching from the particle trajectory map, is capable of drastically inflating the norm at a later time.  In Section \ref{Section Proof of main theorem}, we iterate this construction to build a countable collection of solutions satisfying norm inflation, and we patch these local solutions together in such a way that they have little pairwise interaction. The full velocity solution is then shown to blow up in $H^{2,\alpha}$-norm instantaneously. 

Our approach most significantly differs from the approaches of \cite{bourgainli15, bourgainli21, kwon21} in that we must establish existence of initial vorticity which is small in the $H^{1,\alpha}$-seminorm and which leads to norm inflation. To show that $|\omega_0|_{H^{1,\alpha}}$ can be made small, we utilize a real-variable Gagliardo-type equivalent seminorm (see (\ref{real variable definition for s=1})), as well as a logarithmic interpolation inequality (see Lemma \ref{interpolation lemma}).  The proof of $H^{1,\alpha}$-norm inflation requires delicate frequency analysis using the Fourier definition of the $H^{1,\alpha}$-seminorm, as given in (\ref{Logarithmic Sobolev seminorm}).

Finally, we remark that after suitable modification, many of our arguments apply to a more generally regularized critical Sobolev space; that is, where the logarithmic derivative is replaced with a generic function which grows slower than a power at infinity. These more general spaces are considered by the second and third authors in \cite{harrisonradke25}, where the authors specify a threshold for the growth rate of the Fourier multiplier, above which global well-posedness is guaranteed. In future work, we plan to address the ill-posedness question on the other side of this threshold. %\textbf{I am not sure we need to specify the obstacles encountered when generalizing to a bigger class of Sobolev spaces.  There is no real benefit.  So I have commented out that part for now.} %When applying the arguments in this work to the generalized spaces, we encounter two main difficulties; first, we must obtain a suitable replacement for the interpolation inequality (see Lemma \ref{interpolation lemma}), and, second, the generalized spaces lack an equivalent real-variable seminorm (see (\ref{real variable definition for s=1})), an important tool in verifying inflation of the $H^{1,\alpha}$-norm in Section \ref{section Local Norm Inflation}. 

%%%%%%%%%%%%%%%%%%%%%%%%%%%%%%%%%
%%%%%%%%%%%%%%%%%%%%%%%%%%%%%%%%%
%%%%%%%%%%%%%%%%%%%%%%%%%%%%%%%%%
\section{Preliminaries and Notation}\label{Section Preliminaries and Notation}
%%%%%%%%%%%%%%%%%%%%%%%%%%%%%%%%%
%%%%%%%%%%%%%%%%%%%%%%%%%%%%%%%%%
%%%%%%%%%%%%%%%%%%%%%%%%%%%%%%%%%
\subsection{Notation} In what follows, we let $\log$ denote the base-two logarithm, and we utilize the notation $\log^\alpha(t)$ to denote $\left(\log(t)\right)^\alpha$ for $\alpha\in\mathbb{R}$, $t>0.$ We denote the open disc of radius $r>0$ around a point $x\in\mathbb{R}^2$ by $B_r(x). $ For a matrix $A$, $|A|_\infty$ denotes the largest absolute value of the entries of a matrix $A$.  Finally, we let $|\cdot|$ denote the standard Euclidean norm in $\mathbb{R}^2$; that is, for $x=(x_1,x_2)\in\mathbb{R}^2$, $|x|=\sqrt{x_1^2+x_2^2}$.

Throughout this work, we use $C>0$ to denote a constant which may change from one line to the next.  We may indicate that a constant depends on some quantity if the dependence on that quantity is important to the argument; for example, if $C$ depends on $\alpha$, we may utilize the notation $C_{\alpha}$.  Often, to simplify the presentation, we apply the notation $A \lesssim B$ to indicate that $A\leq C B$ for some constant $C>0$.      

\subsection{Sobolev Spaces} 
Let $k$ be a non-negative integer. The {\em Sobolev space} $W^{k,p}(\mathbb{R}^d)$ is the space of distributions $f$ in $L^p(\mathbb{R}^d)$ whose distributional derivatives of all orders up to and including $k$ are also in $L^p(\mathbb{R}^d)$.  We equip $W^{k,p}(\mathbb{R}^d)$ with the norm $$\|f\|_{W^{k,p}}=\|f\|_{L^p}+\sum_{|\sigma|=1}^k\|D^\sigma f\|_{L^p}.$$
When $p=2$, we denote the $k$th order Sobolev space $W^{k,2}(\mathbb{R}^d)$ by $H^k(\mathbb{R}^d)$.

When $p\in(1,\infty)$, one can equivalently define $W^{k,p}(\mathbb{R}^d)$ to be the set of distributions $f$ in $L^p(\mathbb{R}^d)$ such that $(-\Delta)^{k/2}f\in L^p(\mathbb{R}^d)$, where $(-\Delta)^{k/2}$ is the fractional Laplacian, with norm $$\|f\|_{W^{k,p}}=\|(Id-\Delta)^{k/2}f\|_{L^p}.$$ This definition allows for non-integral values of $k$, enabling us to make sense of Sobolev spaces $W^{k,p}(\mathbb{R}^d)$ for any differentiability index $k\in\mathbb{R}$. When $p=2$ and $k\in\mathbb{R}$, we can apply Plancherel's identity and obtain $$\|f\|_{H^k} \simeq \left(\int_{\mathbb{R}^d}(1+|\xi|^2)^{k}|\widehat{f}(\xi)|^2\, d\xi\right)^{1/2},$$ where $\simeq$ denotes equivalence of norms and $\widehat{f}:\mathbb{R}^d\to\mathbb{R} $ denotes the Fourier transform of $f$ defined by $$\widehat{f}(\xi)=\int_{\mathbb{R}^d}f(x)e^{-ix\cdot\xi}\, d\xi.$$ 

In what follows, we make use of the Sobolev Embedding Theorem, which states that if $1< p\leq q< \infty$, $p<d$, and $s\geq t\geq0$, and if $s,p,t,q$ satisfy $$\frac{1}{p}-\frac{s}{d}=\frac{1}{q}-\frac{t}{d},$$ then there exists a constant $C>0$ such that $\|f\|_{W^{t,q}}\leq C\|f\|_{W^{s,p}}$ for all $f\in W^{s,p}(\mathbb{R}^d)$. 

We also recall the following interpolation inequality: if $p\in[1,\infty]$ is fixed, when $s_1>s>s_2\geq 0$ and $\theta\in(0,1)$ satisfy $s = \theta s_1 +(1-\theta)s_2,$ there exists $C>0$ such that $$\|f\|_{W^{s,p}}\leq C\|f\|_{W^{s_1,p}}^\theta\|f\|_{W^{s_2,p}}^{1-\theta}$$ for all $f\in W^{s_1,p}(\mathbb{R}^d)$. 

\subsection{Logarithmic Sobolev Spaces} For $\alpha\in\mathbb{R}$, we define the {\em $\alpha$-logarithmic derivative} of a distribution $f$ by 
$$\log^\alpha(D+1)f(x) := \int_{\mathbb{R}^2}\log^\alpha(|\xi|+1)\widehat{f}(\xi)e^{ix\cdot\xi}\ d\xi.$$
For $p\in (1,\infty),s\geq 0$, and $\alpha\geq0$, we define the \textit{logarithmic Sobolev space} $W^{s,p,\alpha}(\mathbb{R}^d)$ to be the set of distributions $f\in L^p(\mathbb{R}^d)$ such that $(-\Delta)^{s/2}\log^{\alpha}(D+1)f$ belongs to $L^p(\mathbb{R}^d)$. We define the seminorm $$|f|_{W^{s,p,\alpha}}:=\|(-\Delta)^{s/2}\log^{\alpha}(D+1)f\|_{L^p}.$$
Then the norm on $W^{s,p,\alpha}(\mathbb{R}^d)$ is defined as $$\|f\|_{W^{s,p,\alpha}}:=\|f\|_{W^{s,p}}+|f|_{W^{s,p,\alpha}}.$$ %=\|f\|_{W^{s,p}}+\|\log^\alpha(D+1)\Delta^{s/2}f\|_{L^p}.$$ 
When $p=2$, we set $H^{s,\alpha}(\mathbb{R}^d):=W^{s,2,\alpha}(\mathbb{R}^d).$ Note that by Plancherel's identity, $\| f \|_{H^{s,\alpha}} \simeq \|f\|_{H^s}+|f|_{H^{s,\alpha}}$, with the seminorm $|f|_{H^{s,\alpha}}$ defined as in (\ref{Logarithmic Sobolev seminorm}).

For simplicity, in this work we restrict our attention to the $L^2$-based spaces $H^{s,\alpha}(\mathbb{R}^2)$, though we expect Theorem \ref{main} to hold for initial velocity in $W^{s,p,\alpha}(\mathbb{R}^2)$ with $s=2/p+1$, $p\in (1,\infty)$, and $\alpha\in(0,1-\frac{1}{p}].$

In the proof of Theorem \ref{main}, we make use of several important properties of logarithmic Sobolev spaces.  In Lemma \ref{interpolation lemma} below, we establish the first of these properties, a key logarithmic interpolation inequality which relates logarithmic Sobolev spaces to the classical Sobolev spaces. Below, we let $|f|_{H^s}$ denote the homogeneous Sobolev norm; that is, $$|f|_{H^s}^2=\int_{\mathbb{R}^d}|\xi|^{2s}|\widehat{f}(\xi)|^2\, d\xi.$$
\begin{lemma}\label{interpolation lemma}
    Let $\alpha\in(0,1)$ and $s>t\geq 0$. Then there exists $C>0$, depending on $\alpha$ and $s-t$, such that for any $f\in H^s(\mathbb{R}^d)$, 
    \begin{equation}\label{log sobolev interpolation inequality}
    |f|_{H^{t,\alpha}}\leq C|f|_{H^{t}}\log^{\alpha }\left(\frac{\|f\|_{H^s}}{|f|_{H^t}}\right).
\end{equation}
Moreover, if $f$ satisfies $$0<|f|_{H^{t}}\leq a, \quad \|f\|_{H^{s}}\leq b,$$ then
\begin{equation*}|f|_{H^{t,\alpha}}\leq Ca\log^\alpha \left(\frac{b}{a}+1\right).
\end{equation*}
\end{lemma}
\begin{proof}
Let $d\mu$ denote the probability measure $$d\mu(\xi):=|f|_{H^t}^{-2}|\xi|^{2t}|\widehat{f}(\xi)|^2d\xi.$$ Then by Jensen's inequality,
 \Obsolete{   \begin{equation}\label{Jensens}
    \begin{split}
    |f|_{H^{t,\alpha}}^2&=(2s-2t)^{-{2\alpha}}|f|_{H^t}^2\int_{\mathbb{R}^2}\log^{2\alpha}\left((1+|\xi|)^{2(s-t)}\right)\,d\mu(\xi)
    \\ &\leq (2s-2t)^{-{2\alpha}}|f|_{H^t}^2\log^{2\alpha}\left(\int_{\mathbb{R}^2}(1+|\xi|)^{2(s-t)}\,d\mu(\xi)+1\right)
    \\ &\leq(2s-2t)^{-2\alpha}|f|_{H^t}^2\log^{2\alpha}\left(\frac{\int(1+|\xi|)^{2s}|\widehat{f}(\xi)|^2d\xi}{|f|_{H^t}^2}+1\right).
    \end{split}
    \end{equation}
    \textbf{Elaine's thoughts - I am not sure the $+1$ in the second line above is needed.  See my attempt at a revision below.}
Then by Jensen's inequality  } %%%end obsolete
    \begin{equation}\label{Jensens}
    \begin{split}
    |f|_{H^{t,\alpha}}^2&=(2s-2t)^{-{2\alpha}}|f|_{H^t}^2\int_{\mathbb{R}^d}\log^{2\alpha}\left((1+|\xi|)^{2(s-t)}\right)\,d\mu(\xi)
    \\ &\leq (2s-2t)^{-{2\alpha}}|f|_{H^t}^2\log^{2\alpha}\left(\int_{\mathbb{R}^d}(1+|\xi|)^{2(s-t)}\,d\mu(\xi)\right)
    \\ &\leq (2s-2t)^{-2\alpha}|f|_{H^t}^2\log^{2\alpha}\left(\frac{\int_{\mathbb{R}^d}(1+|\xi|)^{2s}|\widehat{f}(\xi)|^2 d\xi}{|f|_{H^t}^2}\right).
   % &= (s-t)^{-2\alpha}|f|_{H^t}^2\log^{2\alpha}\left(\frac{\|f\|_{H^s}}{|f|_{H^t}}\right),
    \end{split}
    \end{equation}  
    %where we make sense of the right hand side of the first inequality above by noting that $$\int_{\mathbb{R}^2}(1+|\xi|)^{2(s-t)}\,d\mu(\xi) \geq \frac{|f|^2_{H^s}}{|f|^2_{H^t}} \geq 1.$$
  %  Taking the square root of both sides gives
   % $$
    %|f|_{H^{t,\alpha}}\leq C|f|_{H^{t}}\log^{\alpha }\left(\frac{\|f\|_{H^s}}{|f|_{H^t}}\right).
%$$
Applying the series of inequalities  
\begin{equation*}\label{ineq1}
(1+|\xi|)^{2s}=(1+2|\xi|+|\xi|^2)^{s} \leq  (1+2+\frac{1}{2}|\xi|^2+|\xi|^2)^{s}\leq 3^{s}(1+|\xi|^2)^{s}
\end{equation*}
%and for all $a\geq 1$ and $b\geq 0$, 
%\begin{equation}\label{ineq2}
%\log(ab^2+1)\leq \log\left(a(b+1)^2\right)\leq\log\left((b+1)^{2a}\right)=2a\log(b+1).
%\end{equation}
%Applying (\ref{ineq1}) and (\ref{ineq2}) 
to the integrand on right hand side of (\ref{Jensens}) and taking the square root of both sides yields the desired interpolation inequality,  
\begin{equation*}
    |f|_{H^{t,\alpha}}\leq C|f|_{H^{t}}\log^{\alpha }\left(\frac{\|f\|_{H^s}}{|f|_{H^t}}\right).
\end{equation*}
This completes the proof of (\ref{log sobolev interpolation inequality}).  

As (\ref{log sobolev interpolation inequality}) implies that $|f|_{H^{t,\alpha}}\leq C|f|_{H^{t}}\log^{\alpha }\left(\frac{\|f\|_{H^s}}{|f|_{H^t}}+1\right)$, to complete the proof of Lemma \ref{interpolation lemma}, we observe that whenever $\alpha\in (0,1),$ the function $h:(0,\infty)\to [0,\infty)$ defined by $$h(x)=x\log^{\alpha}\left(\frac{1}{x}+1\right)$$
is nondecreasing.  
%%%%begin obsolete
\Obsolete{Differentiating $h$ gives
\begin{equation*}
\begin{split}
&h'(x) = \log^{\alpha}\left( \frac{1}{x}+1\right) + x\alpha \log^{\alpha-1}\left( \frac{1}{x}+1\right)\left(\frac{1}{1+x} - \frac{1}{x}\right)\\
&\qquad = \log^{\alpha-1}\left( \frac{1}{x}+1\right)\left(\log\left( \frac{1}{x}+1\right) - \frac{\alpha}{1+x} \right).
\end{split}
\end{equation*}
Since $\log^{\alpha-1}\left( \frac{1}{x}+1\right)>0$ on $(0,\infty)$, it suffices to prove that \begin{equation}\label{positiveder}
G(x) = \log\left( \frac{1}{x}+1\right) - \frac{\alpha}{1+x} >0
\end{equation}on $(0,\infty)$.  

To prove (\ref{positiveder}), note first that \begin{equation}\label{Glimits}
\begin{split}
&\lim_{x\rightarrow 0+} G(x) = +\infty, \qquad \lim_{x\rightarrow +\infty} G(x) = 0.
\end{split}
\end{equation}
Moreover, one has
\begin{equation*}
\begin{split}
& G'(x) = \frac{1}{1+x} - \frac{1}{x} + \frac{\alpha}{(1+x)^2} = \frac{x(1+x) - (1+x)^2 +\alpha x }{x(1+x)^2}=\frac{x(\alpha-1) -1}{x(1+x)^2} < 0,
%&\qquad = \log^{\alpha-1}\left( \frac{C}{x}+1\right)\left(\log\left( \frac{C}{x}+1\right) - \frac{C\alpha}{C+x} \right).
\end{split}
\end{equation*}
where we used $x(\alpha-1)<0$ to obtain the last inequality.  As $G$ is continuous and decreasing on $(0,\infty)$ and satisfies (\ref{Glimits}), it follows that $G>0$ on $(0,\infty)$.  Thus (\ref{positiveder}) holds.  We conclude that $h'>0$ and $h$ is increasing, completing the proof of Lemma \ref{interpolation lemma}. } %%%%%%end obsolete   
\end{proof}

In the proof of Theorem \ref{main}, we also make use of the real-variable definition of the $H^{s,\alpha}$- norm. We refer to Theorem 1.4 of \cite{bruenguyen19} for the equivalence 
\begin{equation}\label{logarithmic sobolev real variable definition}
    \|f\|_{H^{s,\alpha}}\simeq \|f\|_{H^s}+\int_{\mathbb{R}^d}\int_{B_{1/2}(x)}\frac{|(-\Delta)^{s/2}f(x)-(-\Delta)^{s/2}f(y)|^2}{|x-y|^{d}\log^{1-2\alpha}\left(|x-y|^{-1}\right)}\, dy\, dx.
\end{equation}
 This equivalent representation becomes particularly useful when $s=1$ and $d=2$, as one can verify using (\ref{logarithmic sobolev real variable definition}) that
 \begin{equation}\label{real variable definition for s=1}\|f\|_{H^{1,\alpha}}\simeq\|f\|_{H^1}+\int_{\mathbb{R}^2}\int_{B_{1/2}(x)}\frac{|\nabla f(x)-\nabla f(y)|^2}{|x-y|^2\log^{1-2\alpha}\left(|x-y|^{-1}\right)}\, dy\, dx.
 \end{equation}

\subsection{Oscillatory Integral Estimate} As mentioned in Section \ref{Section Introduction}, to show ill-posedness in $H^{2,\alpha}(\mathbb{R}^2)$, we construct an initial vorticity which produces a large Lagrangian deformation in short time, and we perturb this initial vorticity by a highly oscillatory function.  To prove estimates on the composition of the perturbation with the particle trajectory map, we make use of Lemma \ref{nonstationary phase lemma} below.  One can find similar types of bounds on oscillatory integrals in, for example, \cite{stein93}.  We provide a detailed proof of Lemma \ref{nonstationary phase lemma}, as we must carefully track how the bounds on the oscillatory integral depend on its phase.    

\begin{lemma}\label{nonstationary phase lemma}
    Let $R>0$, and assume that $f:\mathbb{R}^2\to\mathbb{R}$ is smooth and supported in $B_R(0)$. Also assume $\phi:\mathbb{R}^2\to\mathbb{R}$ is smooth, and suppose that for every $x\in supp (f)$, 
    \begin{equation}\label{gradphi}
    |\nabla\phi(x)|_\infty\geq \lambda>0.
    \end{equation}
    Then $$\left|\int_{\mathbb{R}^2}f(x)e^{-i\phi(x)}\,dx\right|\leq C_N\|\nabla\phi\|^N_{C^N}\lambda^{-2N}\|f\|_{W^{N,1}}\left(1+\frac{R\|\nabla\phi\|_{C^1}}{\lambda}\right)^{N+4} ,$$ for a constant $C_N$ which remains bounded as long as $\|\nabla \phi\|_{C^N}$ remains bounded.  
   % \textbf{It does seem we need at least one more round of IBP.  But after first round, I get:}
  %  \begin{equation*}
  %  \begin{split}
    %&\left|\int_{\mathbb{R}^2}f(x)e^{-i\phi(x)}\,dx\right|\leq \frac{C\|\nabla\phi\|_{C^1}}{A^2}\|f\|_{W^{1,1}}\left(1+ \frac{R\|\nabla\phi\|_{C^1}}{A} \right)^5.
%    \end{split}
%    \end{equation*}
\end{lemma}
\begin{proof}
    Set $$r=\frac{\lambda}{2\|\nabla\phi\|_{C^1}}.$$ It follows from (\ref{gradphi}) and the Mean Value Theorem that for each $x\in supp(f)$, either $|\partial_{x_1}\phi(y)|\geq \lambda/2$ for all $y\in B_r(x)$ or $|\partial_{x_2}\phi(y)|\geq \lambda/2$ for all $y\in B_r(x)$. Since $f$ has compact support, we can find a set of points $X=\{x_k:k=1,...,n\}\subset\mathbb{R}^2$ such that $$supp(f) \subseteq \cup_{k=1}^n B_{r/2}(x_k).$$ %We let $$\mathcal{B}=\{B_{r/2}(x_k):x_k\in X\}.$$ %and write  $\mathcal{B}=\mathcal{B}_1\cup\mathcal{B}_2$, where $B_{r/2}(x_k)\in \mathcal{B}_j$ if $|\partial_{x_j}\phi(y)|\geq A/2$ for all $y\in B_{r/2}(x_k)$. 
    
    We now apply a standard procedure to construct a partition of unity subordinate to the slightly modified open cover $$ \mathcal{B} = \{B_{r}(x_k):x_k\in X\}.$$  For fixed $k$, $1\leq k\leq n$, and for $x\in\mathbb{R}^2$, set 
    $$h_k(x) = \chi \left(\frac{x-x_k}{r}\right),$$
    where $\chi$ is a smooth radial bump function which is identically one on $B_{1/2}(0)$ and vanishes on $B_1(0)^c$.  Note that, by construction, $h_k (z)= 1$ for all $z\in B_{r/2}(x_k)$. Now, for $x\in\mathbb{R}^2$, set 
    \begin{equation*}\label{partition}
    g_k (x) =
\begin{cases}
      \frac{h_k(x)}{\sum_{l=1}^n h_l(x)}, \quad & x\in \cup_{k=1}^n B_{r}(x_k)\\
       0, \quad &\text{otherwise}.
    \end{cases}    
\end{equation*}
    Then $\mathcal{U}=\{g_k:1\leq k\leq n\}$ is clearly a partition of unity subordinate to $\mathcal{B}$. Moreover, for each $x\in supp(f)$, there exists $l$ between $1$ and $n$ such that $x\in B_{r/2}(x_l)$, ensuring that $\sum_{l=1}^n h_l(x)\geq 1$ for all $x\in \cup_{k=1}^n B_{r}(x_k)$.  
    
     We distribute the positive integers $k=1, \dots, n$ among two sets $K_1$ and $K_2$, where $$K_j = \{k: |\partial_{x_j}\phi(y)|\geq \lambda/2 \text{ for all }y\in B_{r/2}(x_k) \}, $$ with the understanding that $K_1\cap K_2$ may be nonempty. Now, since for $1\leq k\leq n$ and $1\leq j\leq 2$, $fg_ke^{-i\phi(x)}=\frac{fg_k}{-i\partial_{x_j}\phi(x)}\partial_{x_j}e^{-i\phi(x)}$, using the identity $f=\sum_{k=1}^n fg_k$ and integration by parts, we may write \begin{equation}\label{oscillatory integral first ibp}
        \begin{split}
          &\left|  \int_{\mathbb{R}^2}f(x)e^{-i\phi(x)}\,dx\right|\leq\sum_{j=1}^2\sum_{k\in K_j} \left|\int_{B_{r}(x_k)}\frac{f(x)g_k(x)}{\partial_{x_j}\phi(x)}\partial_{x_j}e^{-i\phi(x)}\,dx\right|
            \\ & \qquad\qquad = \sum_{j=1}^2\sum_{k\in K_j}\left|\int_{B_{r}(x_k)}\frac{\partial_{x_j}\left(f(x)g_k(x)\right)\partial_{x_j}\phi(x)-f(x)g_k(x)\partial_{x_j}^2\phi(x)}{\left(\partial_{x_j}\phi(x)\right)^2}e^{-i\phi(x)}\,dx\right|,
        \end{split}
    \end{equation}  
    where the inequality results from the possibility that $K_1\cap K_2\not= 0$.
    Since $|\partial_{x_j}\phi(x)|\geq \lambda/2$ on $B_r(x_k)$ when $k\in K_j$, it follows that 
    \begin{equation}\label{oscest}
    \begin{split}
    &\left|\int_{\mathbb{R}^2}f(x)e^{-i\phi(x)}\,dx\right|\leq \frac{Cn}{\lambda^2}\sup_{1\leq k \leq n}\left(\|f\|_{W^{1,1}}\| g_k\|_{C^1}\|\nabla \phi\|_{\infty} +\|f\|_{L^1}\|g_k\|_{L^\infty}\|\nabla \phi\|_{C^1}\right)\\
    &\qquad \leq \frac{Cn}{\lambda^2}\| \nabla\phi\|_{C^1}\|f\|_{W^{1,1}}\sup_{1\leq k \leq n} \| g_k\|_{C^1}.
    %&\qquad \leq |\mathcal{B}_2|\frac{C}{A}\sup_{1\leq k \leq n}  \left(\|f\|_{L^1}\|\nabla g_k\|_\infty +\|\nabla f\|_{L^1}\|g_k\|_{\infty}+\frac{2}{A}\|f\|_{L^1}\|g_k\|_{L^\infty}\|\nabla \phi\|_{C^1}\right)
    \end{split}
    \end{equation}
    To complete the proof of Lemma \ref{nonstationary phase lemma} for $N=1$, we must bound $n$ and $\|g_k\|_{C^1}$.  Since $\sum_{l=1}^n h_l(x)\geq 1$ for all $x\in \cup_{k=1}^n B_{r}(x_k)$, an application of the quotient rule gives
    \begin{equation}\label{gfirstder}
    \| \nabla g_k \|_{\infty} \leq \frac{Cn}{r} \leq \frac{Cn\|\nabla\phi\|_{C^1}}{\lambda} 
    \end{equation}
    for each $k$.  Moreover, since $supp(f) \subseteq B_R(0)$, we can choose $n$ to satisfy $$n\leq C\left(\frac{R}{r}+1\right)^2 \leq C\left(\frac{R\|\nabla\phi\|_{C^1}}{\lambda}+1\right)^2.$$ Substituting these estimates into (\ref{oscest}) gives
    \begin{equation*}
    \begin{split}
    &\left|\int_{\mathbb{R}^2}f(x)e^{-i\phi(x)}\,dx\right|\leq \frac{C\|\nabla\phi\|_{C^1}}{\lambda^2}\|f\|_{W^{1,1}}\left(1+ \frac{R\|\nabla\phi\|_{C^1}}{\lambda} \right)^5.
    \end{split}
    \end{equation*}
    This verifies the lemma for $N=1$. Now, returning to \eqref{oscillatory integral first ibp}, we once again integrate the right hand side by parts and obtain, in a manner similar to that above, the inequality 
      \begin{equation}\label{oscest3}
    \left|\int_{\mathbb{R}^2}f(x)e^{-i\phi(x)}\,dx\right|\leq \frac{Cn}{\lambda^4}\|\nabla\phi\|^2_{C^2}\|f\|_{W^{2,1}}\sup_{1\leq k \leq n}\|g_k\|_{C^2}.
    \end{equation}
    Again by the quotient rule, for multi-indices $|\beta|=2$,
    \begin{equation}\label{gsecondder}
\|\partial^{\beta}g_k\|_{\infty} \leq \frac{Cn^2}{r^2}.
\end{equation}
    We substitute the above bound for $n$ into (\ref{gsecondder}) and apply the resulting inequality and (\ref{gfirstder}) to (\ref{oscest3}), which gives
    $$\left|\int_{\mathbb{R}^2}f(x)e^{-i\phi(x)}\,dx\right| \leq \frac{C\|\nabla\phi\|^2_{C^2}}{\lambda^4}\|f\|_{W^{2,1}}\left(1+ \frac{R\|\nabla\phi\|_{C^1}}{\lambda} \right)^6.$$ Iterating this process $N$ times yields the result.
   % \begin{equation}\label{oscest2}
    %\left|\int_{\mathbb{R}^2}f(x)e^{-i\phi(x)}\,dx\right|\leq c_N\|f\|_{W^{N,1}}\|\nabla\phi\|_{C^N}A^{-N}n \sup_{1\leq k\leq n}\|g_k\|_{C^N}.
    %\end{equation}
    %Substituting the above estimates for $n$ and $\|g_k\|_{C^1}$ into (\ref{oscest2}) yields the desired estimate.
\end{proof}

%%%%%%%%%%%%%%%%%%%%%%%%%%%%%%%%%
%%%%%%%%%%%%%%%%%%%%%%%%%%%%%%%%%
%%%%%%%%%%%%%%%%%%%%%%%%%%%%%%%%%
\section{Local Norm Inflation}\label{section Local Norm Inflation}
%%%%%%%%%%%%%%%%%%%%%%%%%%%%%%%%%
%%%%%%%%%%%%%%%%%%%%%%%%%%%%%%%%%
%%%%%%%%%%%%%%%%%%%%%%%%%%%%%%%%%

In this section, we show that there exists smooth initial data $\omega_0$ and a smooth function $\beta$, both small in $H^{1,\alpha}$-norm, such that the perturbed initial data $\omega_0+\beta$ exhibits local norm inflation; specifically, $\omega_0+\beta$ generates a solution to (\ref{VorticityEquation}) whose $H^{1,\alpha}$-norm can be made arbitrarily large on the time interval $[0,1]$.  

To prove local norm inflation, we construct initial data that yields a solution with large Lagrangian deformation matrix $\nabla X^t$.  We then perturb the initial data by a highly oscillatory function localized around the point at which the large deformation occurs. We will show that, if the frequency of oscillation is large enough, then the solution is well-approximated in $H^{1,\alpha}$-norm by the propagation of the perturbation $\beta$ by the particle trajectory map of the original solution. The Lagrangian stretching of the oscillation then yields the norm inflation.

We complete this process in two parts, by first showing in Theorem \ref{localnorminflationtheorem} that under suitable conditions such a perturbation exists, and then showing in Theorem \ref{largedeformationexistence} that the conditions of Theorem \ref{localnorminflationtheorem} can be met.

 \begin{theorem}\label{localnorminflationtheorem}
     Let $\omega_0\in C^\infty_c(\mathbb{R}^2)$ satisfy the following properties:\\
     \indent (i) there exists $R_0>0$ such that supp$(\omega_0)$$\subset B_{R_0}(0)$, and
    \\ \indent (ii) there exists $t_0\in [0,1]$ such that $$\max_{0\leq t\leq 1}\|\nabla X^t\|_{L^\infty(B_{R_0}(0))}=\|\nabla X^{t_0}\|_{L^\infty(B_{R_0}(0))}=L>2^{13},$$ where $X^t=(X^t_1,X^t_2)$ denotes the particle trajectory map corresponding to the solution, $\omega$, of (\ref{VorticityEquation}) with initial data $\omega_0$, i.e. $$\frac{dX^t}{dt}=(\nabla^\perp(-\Delta)^{-1}\omega)(X^t(\alpha),t),\quad X^0(\alpha)=\alpha,\quad (\alpha,t)\in\mathbb{R}^2\times [0,\infty).$$ 
    Then there exists a function $\beta\in C^\infty_c(\mathbb{R}^2)$ such that the solution $\widetilde{\omega}$ to (\ref{VorticityEquation}) with initial data $\widetilde{\omega_0}=\omega_0+\beta$ satisfies \begin{equation}\label{largeH1alphanormconclusion}
    \|\widetilde{\omega}(t_0)\|_{H^{1,\alpha}}>L^{\frac{1}{3}}.\end{equation}
    Moreover, $\beta$ satisfies    
\begin{equation}\label{smallperturb} 
\begin{split}
&\text{supp}(\beta)\subset B_{R_0}(0),
\\
&\|\beta\|_{L^p}\leq \|\omega_0\|_{L^p} \quad \text{ for all  } p\in[1,\infty], \text{ and }
%\\ &\|\beta\|_\infty\leq \|\omega_0\|_\infty, 
\\ &|\beta|_{H^{1,\alpha}}\leq c_\alpha L^{-\frac{1}{2}},
\end{split}
\end{equation}
for a constant $c_\alpha>0$ which is independent of both $\omega_0$ and $L$.
\end{theorem}
\begin{remark}
The existence of a solution $\omega$ of (\ref{VorticityEquation}) satisfying the assumptions of Theorem \ref{localnorminflationtheorem} is established in Theorem \ref{largedeformationexistence} below.  
\end{remark}
    \begin{proof} 
         Let $\omega$ denote the solution to (\ref{VorticityEquation}) with initial data $\omega_0$. First note that if $\omega$ is such that $\|\omega(t_0)\|_{H^{1,\alpha}}>L^{\frac{1}{3}}$, then we can take $\beta\equiv 0 $ and the theorem holds. Thus, in what follows we assume that \begin{equation}
             \label{OriginalH1alphasmallassumption} \|\omega(t_0)\|_{H^{1,\alpha}}\leq L^{\frac{1}{3}}.
         \end{equation}
         {\bf Construction of $\beta$ satisfying (\ref{smallperturb}).} Observe that by assumption ($ii$) and the smoothness of $ X^t$, there exists $x_0=(x_0^1,x_0^2)$ with $x_0^1x_0^2\not=0$ and $r_0>0$ such that $B_{r_0}(x_0)\subset B_{R_0}(0)$ and such that one component of $\nabla X^{t_0}$ is at least $L/2$ in magnitude on $B_{r_0}(x_0)$.  Without loss of generality, we assume 
        \begin{equation}\label{Lagrangrian Deformation Large in small ball}
            \frac{L}{2}\leq
        \left|\frac{\partial X^{t_0}_2}{\partial x_2}(x)\right|\leq L
        \end{equation}
        for all $x\in B_{r_0}(x_0)$. 
        
        We now construct our highly oscillatory perturbation $\beta$. First, let $\rchi\in C^\infty_c(\mathbb{R}^2)$ be a standard radial bump function with supp$(\rchi)\subset B_1(0)$, $0\leq\rchi(x)\leq1$ for all $x\in\mathbb{R}^2$, and $\rchi(x)=1$ for all $x\in B_{1/2}(0)$. Choose $\delta>0$ so that $\delta\ll \min\{|x_0^1|,|x_0^2|,r_0\}$, and define $$\varphi(x_1,x_2)=\frac{1}{\delta}\sum_{a_1,a_2\in\{-1,1\}}a_2\rchi\left(\frac{x_1-a_1x_0^1}{\delta},\frac{x_2-a_2x_0^2}{\delta}\right).$$ Note that $\varphi$ is an odd function of $x_2$ and an even function of $x_1$. Moreover, by our choice of $\delta$, the support of $\varphi$ is contained in four disjoint balls, each of which is contained in $B_{R_0}(0)$.  One can then check that $\|\varphi\|_{L^2} = 4\|\chi\|_{L^2}$.
        
        Now, for fixed $k\in\mathbb{N}$, define $$\beta(x_1,x_2) = \frac{1}{k\log^\alpha(k+1)\sqrt{L}}\sin(kx_1)\varphi(x_1,x_2).$$  It follows immediately that supp($\beta$) is contained in $B_{R_0}(0)$;  thus, $(\ref{smallperturb})_1$ is satisfied.   
        
        We wish to show that if $k>0$ is chosen sufficiently large, then $\beta$ satisfies the norm bounds in  (\ref{smallperturb}).  To this end, observe that for $p\in[1,\infty]$, \begin{equation}\label{beta Lp estimate}\|\beta\|_{L^p}\leq\frac{\|\varphi\|_{L^p}}{k\log^\alpha(k+1)\sqrt{L}}, %\leq\|\omega_0\|_{L^p},
        \end{equation} 
        so that, for $k$ sufficiently large, $\|\beta\|_{L^p}\leq \|\omega_0\|_{L^p}$ for each $p\in [1,\infty]$.  To establish the bound on the $H^{1,\alpha}$-seminorm of $\beta$, we estimate higher derivatives of $\beta$ in $L^2$-norm and apply (\ref{log sobolev interpolation inequality}).  Note that by the product rule and the chain rule,$$\|\nabla \beta\|_{L^2}\leq \frac{1}{\log^\alpha(k+1)\sqrt{L}}\left(\|\varphi\|_{L^2}+\frac{1}{k}\|\nabla\varphi\|_{L^2}\right),$$
        $$\|\nabla^2\beta\|_{L^2}\leq \frac{1}{\log^\alpha(k+1)\sqrt{L}}\left(k\|\varphi\|_{L^2}+2\|\nabla\varphi\|_{L^2}+\frac{1}{k}\|\nabla^2\varphi\|_{L^2}\right).$$
        So, for large enough $k$, we have  
        \begin{equation}\label{beta H1 estimate}
            \|\nabla \beta\|_{L^2}\leq \frac{2}{\log^{\alpha}(k+1)\sqrt{L}}\|\varphi\|_{L^2}
        \end{equation}
        and
        \begin{equation}\label{beta H2 estimate}
            \|\nabla^2\beta\|_{L^2}\leq \frac{2k}{\log^{\alpha}(k+1)\sqrt{L}}\|\varphi\|_{L^2}
        \end{equation}

Applying Lemma \ref{interpolation lemma} with (\ref{beta H1 estimate}), (\ref{beta H2 estimate}), we obtain, for sufficiently large $k$,
\begin{equation*}
\begin{split}
&|\beta|_{H^{1,\alpha}}\leq C_{\alpha} |\beta|_{H^1} \log^{\alpha} \left( \frac{\|\beta\|_{H^2}}{|\beta|_{H^1}} +1 \right)
\le C_{\alpha}|\beta|_{H^1}\log^\alpha\left( \frac{6k\|\varphi\|_{L^2}}{|\beta|_{H^1}\log^{\alpha}(k+1)\sqrt{L}} +1 \right)\\
&\qquad \le \frac{2C_{\alpha}\|\varphi\|_{L^2}}{\log^{\alpha}(k+1)\sqrt{L}}\log^\alpha\left( 3k +1 \right)\leq \frac{c_\alpha}{\sqrt{L}}.
\end{split}
\end{equation*}
        Thus, (\ref{smallperturb}) is satisfied. \\
     \\   
        \noindent {\bf Proof of (\ref{largeH1alphanormconclusion})} Define $\widetilde{\omega}$ by $$\begin{cases}
            \partial_t\widetilde{\omega}+\widetilde{u}\cdot\nabla \widetilde{\omega} = 0,
            \\ \widetilde{u}=\nabla^\perp(-\Delta)^{-1}\widetilde{\omega},
            \\ \widetilde{\omega}|_{t=0}=\widetilde{\omega}_0:=\omega_0+\beta,
        \end{cases}$$ and $\widetilde{X}^t$ by $$\frac{d}{dt}\widetilde{X}^t(x)=\widetilde{u}(\widetilde{X}^t(x),t),\quad \widetilde{X}^0(x)=x, \quad (x,t)\in\mathbb{R}^2\times [0,\infty).$$
    Define the two functions $$W(\cdot,t)= \omega_0\circ\widetilde{X}^{-t}, \quad p(\cdot,t) = \beta\circ\widetilde{X}^{-t}.$$  Since $\tilde{\omega} = W + p$, it follows from the triangle inequality that    \begin{equation}\label{triangle inequality}
\|\widetilde\omega\|_{H^{1,\alpha}}\geq\|p\|_{H^{1,\alpha}}-\|W\|_{H^{1,\alpha}}.
\end{equation}
To establish (\ref{largeH1alphanormconclusion}), we will first show that $\|W\|_{H^{1,\alpha}}$ can be approximated by $\|\omega\|_{H^{1,\alpha}}$.  Specifically, we establish (\ref{Womegaapprox1}) below.  Our assumption (\ref{OriginalH1alphasmallassumption}) then gives an upper bound for $\|W(t_0)\|_{H^{1,\alpha}}$ in terms of $L$.  This upper bound, combined with (\ref{triangle inequality}), reduces the proof of (\ref{largeH1alphanormconclusion}) to showing that $\|p(t_0)\|_{H^{1,\alpha}}$ is sufficiently large.
\\
    
  \noindent {\bf Estimate for $\| W -\omega\|_{H^{1,\alpha}}$.} We will show that, for sufficiently large $k$, \begin{equation}\label{Womegaapprox1}
  \max_{0\leq t\leq 1}\| W -\omega\|_{H^{1,\alpha}}<1.
  \end{equation}
  First note that $W-\omega$ satisfies
  \begin{equation*}\label{Womegaapprox}
  \partial_{t}(W-\omega)+(\widetilde{u}-u)\cdot \nabla W+u\cdot \nabla(W-\omega)=0.
  \end{equation*}
  We multiply this equality by $W-\omega$, integrate over $\mathbb{R}^2$, and apply Holder's inequality to write 
  \begin{equation}\label{W-omega1}
  \begin{split}
  &\frac{1}{2}\frac{d}{dt}\|W-\omega\|_{L^2}^2 = -\int_{\mathbb{R}^2} (\widetilde{u}-u)\cdot \nabla W (W-\omega) \, dx - \int_{\mathbb{R}^2} u\cdot \nabla(W-\omega)(W-\omega) \, dx\\
  &\qquad \leq \|u-\widetilde{u}\|_{L^4}\|\nabla W\|_{L^4}\|W-\omega\|_{L^2}.  \\
  %&\qquad \leq \|\omega-\widetilde{\omega}\|_{L^{4/3}}\|\nabla W\|_{L^4}\|W-\omega\|_{L^2}\\
  %&\qquad \leq \|\omega-\widetilde{\omega}\|_{L^2}\|\nabla W\|_{L^4}\|W-\omega\|_{L^2},
  \end{split}
  \end{equation}
  Note that we used the divergence free condition on $u$ to conclude that the second integral above vanishes.  We now utilize the Hardy-Littlewood-Sobolev inequality and the compact support of $\omega$ and $\tilde{\omega}$ to write
 \begin{equation}\label{uomegabound}
 \|u-\widetilde{u}\|_{L^4} \lesssim\|\omega-\widetilde{\omega}\|_{L^{4/3}} \lesssim \|\omega-\widetilde{\omega}\|_{L^{2}},
 \end{equation}
  where the implied constant depends on the measure of the support of $\omega_0$.  Substituting this estimate into (\ref{W-omega1}) gives
  \begin{equation}\label{omega pregronwall}
  \begin{split}
  &\frac{d}{dt}\|W-\omega\|_{L^2}^2 \lesssim \|\omega-\widetilde{\omega}\|_{L^2}\|\nabla W\|_{L^4}\|W-\omega\|_{L^2}\\
  &\qquad \lesssim \|\omega-\widetilde{\omega}\|_{L^2}\| W\|_{H^2}\|W-\omega\|_{L^2} \lesssim \|\omega-\widetilde{\omega}\|_{L^2}\|W-\omega\|_{L^2},
  \end{split}
  \end{equation}
  where we applied the Sobolev 
  Embedding Theorem to obtain the second inequality and Lemma \ref{W H2 bound part 2} to obtain the third inequality.
  
 To complete the bound on $\|W-\omega\|_{L^2}$, we must estimate $\|\omega-\widetilde{\omega}\|_{L^2}$.  To this end, note that $\omega-\widetilde{\omega}$ satisfies 
 $$\partial_t(\omega-\widetilde{\omega})+(u-\widetilde{u})\cdot\nabla \omega+\widetilde{u}\cdot\nabla(\omega-\widetilde{\omega})=0.$$ 
 Multiplying by $\omega-\widetilde{\omega}$, integrating in space, and applying Holder's inequality gives 
 $$\frac{d}{dt}\|\omega-\widetilde{\omega}\|_{L^2}^2 \leq 2\|u-\widetilde{u}\|_{L^4}\|\nabla\omega\|_{L^4}\|\omega-\widetilde{\omega}\|_{L^2}-\int_{\mathbb{R}^2}\widetilde{u}\cdot\nabla(\omega-\widetilde{\omega})^2\, dx.$$
 The integral term above is again zero by the incompressibility of $\widetilde{u}$. By (\ref{uomegabound}) and Lemma \ref{nabla w L4 bound},
\begin{equation}
\begin{split}
&\frac{d}{dt}\|\omega-\widetilde{\omega}\|_{L^2}^2\lesssim\|\nabla\omega\|_{L^4}\|\omega-\widetilde{\omega}\|_{L^2}^2 \lesssim \|\omega-\widetilde{\omega}\|_{L^2}^2.
\end{split}
\end{equation}
Gr\"onwall's lemma and (\ref{beta Lp estimate}) then yield \begin{equation}\label{L2 estimate of vorticity difference1}
    \max_{0\leq t\leq 1}\|\omega-\widetilde{\omega}\|_{L^2}\lesssim \|\omega_0-\widetilde{\omega_0}\|_{L^2}\lesssim \frac{1}{k\log^{\alpha}(k+1)}.
\end{equation}       
Substituting (\ref{L2 estimate of vorticity difference1}) into (\ref{omega pregronwall}), integrating in time, and using that $W(0) = \omega(0)$ gives   
\begin{equation}\label{W-omega L2 difference}
    \max_{0\leq t\leq 1}\|W-\omega\|_{L^2}\lesssim \frac{1}{k\log^{\alpha}(k+1)}.
\end{equation}  
Combining Lemma \ref{W H2 bound part 2}, Lemma \ref{W H2 bound}, and (\ref{W-omega L2 difference}) and applying Sobolev interpolation together with the embedding $H^{\frac{3}{2}}(\mathbb{R}^2)\hookrightarrow H^{1,\alpha}(\mathbb{R}^2)$ gives $$\max_{0\leq t\leq 1}\|W-\omega\|_{H^{1,\alpha}}\lesssim\max_{0\leq t\leq 1}\|W-\omega\|_{H^{\frac{3}{2}}}\lesssim\left(\frac{1}{k\log^{\alpha}(k+1)}\right)^{\frac{1}{4}}.$$ 
Thus for sufficiently large $k$, 
\begin{equation}\label{W-omega log sobolev estimate}
\max_{0\leq t\leq 1}\|W-\omega\|_{H^{1,\alpha}}<1.
\end{equation}   
\\
\noindent {\bf Lower bound for $\|p(\cdot,t_0)\|_{H^{1,\alpha}}$.}  Note that the estimate in (\ref{W-omega log sobolev estimate}) and the triangle inequality imply that, for sufficiently large $k$, $$\|\widetilde\omega\|_{H^{1,\alpha}}\geq\|p\|_{H^{1,\alpha}}-\|W\|_{H^{1,\alpha}}\geq\|p\|_{H^{1,\alpha}}-\|\omega\|_{H^{1,\alpha}}-1. $$  We will show that for sufficiently large $k$, 
\begin{equation}\label{plarge}
\|p(\cdot,t_0)\|_{H^{1,\alpha}}\geq \frac{1}{2}L^\frac{1}{2}.
\end{equation}
This estimate, combined with (\ref{OriginalH1alphasmallassumption}), will allow us to conclude that (\ref{largeH1alphanormconclusion}) holds, as (\ref{plarge}) implies
$$\|\widetilde{\omega}(t_0)\|_{H^{1,\alpha}}\geq \frac{1}{2}L^{\frac{1}{2}}-L^{\frac{1}{3}}-1>L^{\frac{1}{3}},$$ when $L>2^{13}.$

 To establish (\ref{plarge}), first define $$q(\cdot,t)=\beta\circ X^{-t}.$$ In what follows, we estimate the size of $q(\cdot,t_0)$ in $H^{1,\alpha}$ and show that, for frequencies near $kL$ with $k$ sufficiently large, $q$ sufficiently approximates $p$ in $L^2$-norm. This approximation will yield the desired lower bound on $\|p(\cdot,t_0)\|_{H^{1,\alpha}}$.
 
 To estimate $\|q(\cdot,t_0)\|_{H^{1,\alpha}}$, we analyze the frequencies of the Fourier transform of $q(\cdot,t_0)$. Denote $$\widehat{q}_{t_0}(\xi)=\int_{\mathbb{R}^2}q(x,t_0)e^{- i x\cdot\xi}\, dx.$$  
 The definition of $\beta$ combined with a change of variables $x\mapsto X^{t_0}(x)$ gives $$\widehat{q}_{t_0}(\xi)=\frac{1}{k\log^{\alpha}(k+1)\sqrt{L}}\int_{\mathbb{R}^2}\sin(kx_1)\varphi(x)e^{- iX^{t_0}(x)\cdot\xi}\, dx.$$ Defining $F(\xi):=k\log^{\alpha}(k+1)\sqrt{L}\widehat{q}_{t_0}(\xi)$, we have 
\begin{equation}\label{F as oscillatory integral}
F(\xi) =\frac{1}{2i}\int_{\mathbb{R}^2}\varphi(x)\left( 
e^{-iX^{t_0}(x)\cdot\xi+ikx_1}-
e^{-iX^{t_0}(x)\cdot\xi-ikx_1} \right)\, dx,
\end{equation}which is a difference of two oscillatory integrals. In order to estimate the size of $F$ in different ranges of frequencies, we apply the method of nonstationary phase. Consider the phases $$\phi^{\pm}_\xi(x):=X^{t_0}(x)\cdot\xi\pm kx_1.$$ Since $u$ is divergence free, $\det\nabla X^{t_0}(x) =1 $, which implies $$\nabla\phi_\xi^\pm(x) = \nabla X^{t_0}(x)\xi\pm ke_1 = \nabla X^{t_0}(x)\left(\xi \pm k\nabla^\perp X_2^{t_0}(x)\right), $$ where $e_1=\begin{pmatrix}
    1\\0
\end{pmatrix}$. Hence we may write $$\xi\pm k\nabla^\perp X_2^{t_0}(x)=\left(\nabla X^{t_0}(x)\right)^{-1}\left(\nabla \phi_\xi^\pm(x)\right).$$ 
Now observe that for every $x\in B_{r_0}(x_0)$, we have $$\frac{L}{2}\leq |\nabla^\perp X_2^{t_0}(x)|\leq \sqrt{2}L,$$ where the lower bound follows from (\ref{Lagrangrian Deformation Large in small ball}) and the upper bound follows from assumption (ii).
It follows that if $|\xi|\leq \frac{1}{4}kL$ and $x\in B_{r_0}(x_0)$, $$|\nabla \phi_\xi^\pm(x)|\geq \frac{1}{2|\left(\nabla X^{t_0}(x)\right)^{-1}|_\infty}|\xi\pm k\nabla^\perp X_2^{t_0}(x)|\geq \frac{1}{2\sqrt{2}L} \frac{kL}{4}\geq \frac{k}{8\sqrt{2}}.$$  
Thus, there exists some $c_0>0$ such that for any $x\in B_{r_0}(x_0)$ and $\xi\in\mathbb{R}^2$ satisfying $|\xi|\leq\frac{1}{4}kL$, we have 
\begin{equation}\label{phase estimate}
|\nabla \phi^\pm_\xi(x)|_\infty\geq c_0k.
\end{equation}
We refer to Lemma \ref{nonstationary phase lemma} to obtain the point-wise estimate on $F$, $$|F(\xi)|=\frac{1}{2}\left|\int_{\mathbb{R}^2}\varphi(x)\left(e^{-i\phi_\xi^-(x)}-e^{-i\phi^+(x)}\right)\,dx\right|\leq C\|\nabla\phi_\xi^\pm\|_{C^2}^2\|\varphi\|_{W^{2,1}}\left(1+\frac{\|\nabla \phi_\xi^\pm\|_{C^1}}{k}\right)^6k^{-4}.$$
Here $\|\nabla\phi^\pm\|_{C^n}:=\max\left\{\|\nabla\phi^-\|_{C^n},\|\nabla\phi^+\|_{C^n}\right\}$ and $C>0$ depends on $c_0$, $r_0$.  Moreover, we have $$\|\nabla\phi_\xi^\pm\|_{C^2}\leq|\xi|\|\nabla X^{t_0}\|_{C^2}\leq \frac{1}{4}kL\|\nabla X^{t_0}\|_{C^2}.$$
Note that the derivatives of $\nabla X^{t_0}$ can be controlled by derivatives of $\omega$, which we can control by the same derivatives of $\omega_0$ by arguments similar to those used in the proofs of Lemmas \ref{nabla w L4 bound} and \ref{W H2 bound part 2}. Thus we obtain $$|F(\xi)|\leq C_1k^{-2}\quad \text{ for all } \xi\in B_{\frac{1}{4}kL}(0),$$ where $C_1>0$ is independent of $k$. It now follows that
\begin{equation}\label{lowhighfreq}
\|F\mathds{1}_{|\xi|\leq \frac{1}{4}kL}\|_{L^2}\lesssim k^{-1}.
\end{equation}
Moreover, since $X^{-t}$ is measure-preserving, an application of Plancherel's identity gives 
\begin{equation}\label{allfreq}
\begin{split}
    & \|F\|_{L^2}^2= \left(2\pi k \log^{\alpha}(k+1)\sqrt{L} \|\beta\|_{L^2} \right)^2=4\pi^2\int_{\text{supp}(\varphi)}|\sin(kx_1)\varphi(x)|^2\, dx 
    \\ &\qquad\geq\frac{16\pi^2}{\delta^2}\int_{|x-x_0|<\delta/2}\left(\frac{1}{2}-\frac{1}{2}\cos(2kx_1) \right) \, dx \\
    &\qquad = 2\pi^2 - \frac{16\pi^2}{\delta^2} \int_{|x-x_0|<\delta/2} \frac{1}{2}\cos(2kx_1) \, dx
    \geq 2\pi^2- C\frac{1}{k}
\end{split}
\end{equation} for some constant $C>0$ which depends on the support of $\varphi$.  In view of (\ref{allfreq}) and (\ref{lowhighfreq}) we have, for sufficiently large $k$, \begin{equation}\label{F midfrequency L2}
    \begin{split}
        \|F\mathds{1}_{|\xi|\geq\frac{1}{4} kL}\|_{L^2} &\geq \|F\|_{L^2}-\|F\mathds{1}_{|\xi|\leq\frac{1}{4}kL}\|_{L^2}
        \geq 4.
    \end{split}
\end{equation}
We now show that $F$ adequately approximates $k\log^{\alpha}(k+1)\sqrt{L}p$ in $L^2$ norm.  Define $$\eta = k\log^{\alpha}(k+1)\sqrt{L}(p-q).$$ Then $\eta$ solves 
\begin{equation}\label{eta}
\begin{cases}\partial_t\eta+k\log^{\alpha}(k+1)\sqrt{L}(\widetilde{u}-u)\cdot \nabla q+\widetilde{u}\cdot \nabla\eta =0, \\ \eta|_{t=0}=0.\end{cases}
\end{equation}
Multiplying $(\ref{eta})_1$ by $\eta$ and integrating over $\mathbb{R}^2$ gives 
\begin{equation}\label{eta L2 growth rate}\frac{d}{dt}\|\eta\|_{L^2}^2\leq 2k\log^{\alpha}(k+1)\sqrt{L}\|\widetilde{u}-u\|_{L^4}\|\nabla q\|_{L^4}\|\eta\|_{L^2}, 
\end{equation}
where we used that div $\widetilde{u}=0$ to conclude that $\int_{\mathbb{R}^2}\eta\widetilde{u}\cdot\nabla\eta \, dx = 0$. We estimate the terms on the right hand side of (\ref{eta L2 growth rate}).  Recall that by (\ref{uomegabound}) and (\ref{L2 estimate of vorticity difference1}), 
\begin{equation}\label{velocity error L4 estimate}
\|\widetilde{u}-u\|_{L^4}\lesssim\|\widetilde{\omega}-\omega\|_{L^2}\lesssim\|\widetilde{\omega_0}-\omega_0\|_{L^2}\lesssim \frac{1}{k\log^\alpha(k+1)}.
\end{equation}
Moreover, since $q$ solves $$\partial_tq+u\cdot\nabla q =0,\quad q|_{t=0}=\beta,$$ we apply an argument identical to that used to derive (\ref{nablaomega}) below to obtain $$\frac{d}{dt}\|\nabla q\|_{L^4}\lesssim \|\nabla u\|_{L^\infty}\|\nabla q\|_{L^4}.$$ Gr\"onwall's inequality provides, for some constant $c>0$ independent of $k$, 
\begin{equation}\label{nablaq}
\|\nabla q\|_{L^4}\lesssim \|\nabla \beta\|_{L^4}\exp\left\{c\int_0^{t_0}\|\nabla u(s)\|_\infty\, ds\right\}.
\end{equation}
By the embedding $W^{1,4}(\mathbb{R}^2)\hookrightarrow C^0_B(\mathbb{R}^2)$, the boundedness of Calderon-Zygmund operators on $L^4(\mathbb{R}^2)$, and the conservation of $L^p$ norms of the vorticity, we have the series of inequalities $$\|\nabla u\|_\infty\lesssim\|\nabla u\|_{W^{1,4}}\lesssim \|\omega\|_{W^{1,4}}=\|\omega_0\|_{L^4}+\|\nabla\omega\|_{L^4}.$$
Substituting this estimate into (\ref{nablaq}) and applying Lemma \ref{nabla w L4 bound} below, we obtain 
\begin{equation}\label{nabla q L4 estimate}\|\nabla q\|_{L^4}\lesssim\|\nabla \beta\|_{L^4}.
\end{equation}
Now by arguing as we did to obtain (\ref{beta H1 estimate}), we have 
$$\|\nabla \beta\|_{L^4}\lesssim \log^{-\alpha}(k+1)L^{-1/2},$$
which, when substituted into (\ref{nabla q L4 estimate}), gives
\begin{equation}\label{nabla q est2}\|\nabla q\|_{L^4}\lesssim\log^{-\alpha}(k+1)L^{-1/2}.
\end{equation}
Substituting (\ref{velocity error L4 estimate}) and (\ref{nabla q est2}) into (\ref{eta L2 growth rate}), integrating in time, and using that $\eta|_{t=0}=0$, we conclude that \begin{equation}\label{scaled p-q estimate}
    \|\eta(t_0)\|_{L^2}\lesssim \log^{-\alpha}(k+1).
\end{equation}
Thus, by (\ref{F midfrequency L2}), (\ref{scaled p-q estimate}), and Plancherel's identity, we may take $k$ sufficiently large to obtain \begin{equation}\label{p frequency estimate}
\begin{split}
&k\log^{\alpha}(k+1)L\|\widehat{p}(t_0)\mathds{1}_{|\xi|\geq\frac{1}{4}kL}\|_{L^2}\geq k\log^{\alpha}(k+1)L\|\widehat{q}(t_0)\mathds{1}_{|\xi|\geq\frac{1}{4} kL}\|_{L^2}-2\pi\sqrt{L}\|\eta(t_0)\|_{L^2}
\\ &\qquad\qquad= \sqrt{L}\left( \|F\mathds{1}_{|\xi|\geq\frac{1}{4} kL}\|_{L^2}-2\pi\|\eta(t_0)\|_{L^2}\right)\geq \sqrt{L}\left(\|F\mathds{1}_{|\xi|\geq\frac{1}{4} kL}\|_{L^2}- 2\right)
\geq 2\sqrt{L}.
\end{split}
\end{equation}
Finally, we are able to estimate the $H^{1,\alpha}$ norm of $p$ as follows: 
\begin{equation}\notag
\begin{split}    &\|p(t_0)\|_{H^{1,\alpha}}^2=\int_{\mathbb{R}^2}|\xi|^2\log^{2\alpha}(|\xi|+1)|\widehat{p}(\xi,t_0)|^2\, d\xi
    \\ &\qquad\geq \frac{1}{16}k^2L^2\log^{2\alpha}\left(\frac{1}{4}kL+1\right)\int_{|\xi|\geq\frac{1}{4} kL} |\widehat{p}(\xi,t_0)|^2\, d\xi
    \\ &\qquad \geq \frac{1}{16}\left(k\log^{\alpha}(k+1)L\|\widehat{p}(t_0)\mathds{1}_{|\xi|\geq \frac{1}{4} kL}\|_{L^2}\right)^2
    \geq \frac{1}{4}L,
\end{split}
\end{equation} where we used that $\frac{1}{4}kL>k$ to obtain the second inequality and (\ref{p frequency estimate}) to obtain the third inequality.  Thus, (\ref{plarge}) holds.

This completes the proof of Theorem \ref{localnorminflationtheorem}.
    \end{proof}

In the proof of Theorem \ref{localnorminflationtheorem} above, we made use of the following three lemmas.

\begin{lemma}\label{nabla w L4 bound}
Assume $\omega$ is a smooth solution to (\ref{VorticityEquation}) with initial data $\omega_0\in C^{\infty}_c(\mathbb{R}^2)$.  Then there exists a constant $C_1>0$ such that   
$$\max_{0\leq t\leq 1}\|\nabla \omega(t)\|_{L^4} \leq C_1.$$
Here $C_1$ depends only on (and increases with respect to) $\|\omega_0\|_{\infty}$, $\|\nabla\omega_0\|_{L^4}$, and the measure of the support of $\omega_0$, 
\end{lemma}
\begin{proof}
The proof is standard, so we only outline the argument.  Applying $\partial_j$ for $j=1,2$ to the vorticity equation, multiplying the resulting equation by $|\partial_j\omega|^2\partial_j\omega$, and integrating over $\mathbb{R}^2$ gives 
\begin{equation*}
\frac{1}{4}\frac{d}{dt}\|\nabla \widetilde{\omega}\|^4_{L^4} + \int_{\mathbb{R}^2} |\partial_j\omega|^2\partial_j\omega (u\cdot\nabla\partial_j\omega) \, dx + \int_{\mathbb{R}^2} |\partial_j\omega|^2\partial_j\omega (\partial_j u \cdot\nabla\omega) \, dx =0.
\end{equation*}
Since $u$ is divergence free, the second term on the left-hand side of the above equality is $0$, while the third term can be bounded using Holder's inequality.  This gives
\begin{equation}\label{nablaomega}
\frac{d}{dt}\|\partial_j {\omega}\|^4_{L^4}\leq C\|\nabla {u}\|_\infty\|\nabla {\omega}\|^4_{L^4}.
\end{equation}
We now substitute the Beale-Kato-Majda logarithmic-Sobolev inequality \cite{bealekatomajda84}, given by 
$$\|\nabla{u}\|_\infty\leq C(1+\|{\omega_0}\|_2+\|{\omega_0}\|_\infty\log\left(2+\|\nabla {\omega}\|_{L^4} \right)),$$ 
into (\ref{nablaomega}).  Applying a standard argument, one can conclude that for each $t>0$,
$$ \| \nabla\omega(t) \|_{L^4} \leq Ce^{e^{Ct\|\omega_0\|_{\infty}}}(1+ \|\nabla\omega_0\|_{L^4})^{e^{Ct\|\omega_0\|_{\infty} }}.$$
Taking the maximum over $t\in[0,1]$ yields the result.
\Obsolete{ %%%%%%%begin obsolete
Hence
    \begin{equation}\label{L4 Gradient perturbed vorticity estimate} 
        \max_{0\leq t\leq 1}\|\nabla \widetilde{\omega}\|_{L^4}\lesssim \|\nabla \widetilde{\omega_0}\|_{L^4}+\|\omega_0\|_\infty.
    \end{equation}
     Note, however, that by the definition of $\beta$, $\|\nabla\beta\|_{L^4}\lesssim\log^{-\alpha}(k+1)$, and thus the implied constant in (\ref{L4 Gradient perturbed vorticity estimate}) can be taken independent of $k$ and only dependent on $\omega_0.$ Similarly, we have \begin{equation}\label{L4 Gradient vorticity estimate}
         \max_{0\leq t\leq 1}\|\nabla \omega\|_{L^4} \lesssim \|\nabla\omega_0\|_{L^4}+\|\omega_0\|_\infty.
     \end{equation}  }
\end{proof}

\begin{lemma}\label{W H2 bound part 2}
Assume $\tilde{\omega}$ is a smooth solution to (\ref{VorticityEquation}) with corresponding velocity $\tilde{u}$ and initial data $\tilde{\omega}_0 = \omega_0 + \beta$, where $\omega_0,\beta \in C^{\infty}_c(\mathbb{R}^2)$.  Further assume $W$ is a smooth solution to the system 
\begin{equation*} 
\begin{cases}
    \partial_t W+\tilde{u}\cdot\nabla W = 0 \quad &\text{on }\mathbb{R}^2\times(0,\infty),
    \\ \tilde{u} = \nabla^\perp(-\Delta)^{-1}\tilde{\omega}  &\text{on }\mathbb{R}^2\times[0,\infty),
    \\ W|_{t=0} =\omega_0 &\text{on }\mathbb{R}^2.
\end{cases}
\end{equation*}
Then there exists a constant $C_0>0$ such that
$$\max_{0\leq t\leq 1}\|W(t)\|_{H^2}\leq C_0.$$
Here $C_0$ depends only on (and increases with respect to) $\|\omega_0\|_{\infty}$, $\|\beta\|_{\infty}$, $\|\omega_0\|_{W^{1,4}}$, $\|\beta\|_{W^{1,4}}$, and the measure of the support of $\omega_0$ and $\beta$.
\end{lemma}
\begin{proof} Since $W$ satisfies the transport equation $$\partial_tW+\widetilde{u}\cdot \nabla W=0,$$ letting $\gamma$ denote a multi-index with $|\gamma|\leq 2$ and applying $\partial^\gamma$ gives $$\partial_t \partial^\gamma W+\widetilde{u}\cdot\nabla \partial^\gamma W=[\widetilde{u}\cdot\nabla, \partial^\gamma]W.$$ Multiplying by $D^\gamma W$ and integrating over $\mathbb{R}^2$ gives $$\frac{1}{2}\frac{d}{dt}\|\partial^\gamma W\|_{L^2}^2=\int_{\mathbb{R}^2}\partial^\gamma W[\widetilde{u}\cdot\nabla,\partial^\gamma ]W\, dx.$$ The triangle and H\"older's inequalities imply that for all $\gamma$ satisfying $|\gamma|\leq 2$, $$\frac{1}{2}\frac{d}{dt}\|\partial^\gamma W\|_{L^2}^2\leq C\|\partial^\gamma W\|_{L^2}\left(\|\widetilde{u}\|_{W^{2,4}}\|\nabla W\|_{L^4}+\|\nabla \widetilde{u}\|_\infty\| W\|_{H^2}\right).$$
Combining this estimate with the conservation $\|W(t)\|_{L^2}=\|\omega_0\|_{L^2}$ for all $t\geq 0$ and summing over $\gamma$ with $|\gamma|\leq 2$, we conclude that
  \begin{equation}\label{W H2 growth rate}
  \frac{d}{dt}\|W\|_{H^2}\leq C(\|\widetilde{u}\|_{W^{2,4}}\|\nabla W\|_{L^4}+\|\nabla \widetilde{u}\|_\infty\| W\|_{H^2}).
  \end{equation} 
By the boundedness properties of Calderon-Zygmund operators and Lemma \ref{nabla w L4 bound}, $$\| \tilde{u}(t)\|_{W^{2,4}}\leq \|\tilde{\omega}(t)\|_{L^4} + \|\nabla\tilde{\omega}(t)\|_{L^4}\leq \|\tilde{\omega}_0\|_{L^4} +Ce^{e^{Ct\|\tilde{\omega}_0\|_{\infty}}}(1+ \|\nabla\tilde{\omega}_0\|_{L^4})^{e^{Ct\|\tilde{\omega}_0\|_{\infty} }}$$ for each $t\in[0,1]$. Also, an application of Sobolev's inequality, boundedness of Calderon-Zygmund operators, and Lemma \ref{nabla w L4 bound} gives 
\begin{equation*}
\begin{split}
&\|\nabla W(t)\|_{L^4}\leq C\|W(t)\|_{H^2},\\
&\|\nabla \tilde{u}(t)\|_\infty\leq C\|\nabla\tilde{u}(t)\|_{W^{1,4}}\leq C\|\tilde{\omega}(t)\|_{W^{1,4}}\leq \|\tilde\omega_0\|_{L^4} + Ce^{e^{Ct\|\tilde\omega_0\|_{\infty}}}(1+ \|\nabla\tilde\omega_0\|_{L^4})^{e^{Ct\|\tilde\omega_0\|_{\infty} }}.
\end{split}
\end{equation*}
Substituting these bounds into (\ref{W H2 growth rate}) and applying Gr\"onwall's inequality provides 
\begin{equation}\label{W H2 max}
    \max_{0\leq t\leq 1}\|W\|_{H^2}\leq C_0,
\end{equation} where $C_0>0$ is a constant dependent only on $\|\omega_0\|_{\infty}$, $\|\beta\|_{\infty}$, $\|\omega_0\|_{W^{1,4}}$, $\|\beta\|_{W^{1,4}}$, and the supports of $\omega_0$ and $\beta$.  
\end{proof}

\begin{lemma}\label{W H2 bound}
Assume $\omega$ is a smooth solution to (\ref{VorticityEquation}) with initial data $\omega_0\in C^{\infty}_c(\mathbb{R}^2)$.  Then there exists a constant $C_0>0$ such that
$$\max_{0\leq t\leq 1}\|\omega(t)\|_{H^2}\leq C_0.$$
Here $C_0$ depends only on (and increases with respect to) $\|\omega_0\|_{\infty}$, $\|\omega_0\|_{W^{1,4}}$, and the measure of the support of $\omega_0$.
\end{lemma}

\begin{proof}
The proof is virtually identical to that of Lemma \ref{W H2 bound part 2}, so we omit the details.
\end{proof}

We spend the remainder of this section constructing initial data that both satisfies the conditions of Theorem $\ref{localnorminflationtheorem}$ and is small in $H^{1,\alpha}(\mathbb{R}^2)$. The initial data resulting from our construction will serve as a sort of building block for the initial vorticity used to generate strong ill-posedness in the proof of the Theorem \ref{main}. 
\begin{theorem}\label{largedeformationexistence}
Given $\epsilon>0$, $\tau>0$, and $\Gamma>0$, there exists $f\in C^{\infty}_{c}(\mathbb{R}^2)$ such that\\
\\
(i) $f(x_1,x_2)$ is odd in $x_1$ and $x_2$ with $f(x_1,x_2)\geq 0$ if $x_1,x_2\geq 0$, \\
(ii) $\| f \|_{\infty} + \| f \|_{H^{1,\alpha}} <\epsilon$,\\
(iii) the smooth, compactly supported solution $\omega$ to (\ref{VorticityEquation}) generated from initial data $\omega_0=f$ has corresponding particle trajectory map $X^t$ satisfying $$\max_{0\leq t\leq \tau} \|\nabla X^t\|_{L^{\infty}(B_{R}(0))} > \Gamma,$$ where $R>0$  satisfies $supp(\omega_0)\subset B_R(0)$.
\end{theorem}
\begin{proof}
Let $\rchi\in C^\infty_c(\mathbb{R}^2)$ be a standard radial bump function such that $0\leq\rchi(x)\leq1$ for all $x\in\mathbb{R}^2$, $\rchi(x)=1$ for all $x\in B_{1/2}(0)$, and supp$(\rchi)\subset B_1(0)$. Define $$\eta(x_1,x_2)=\sum_{a_1,a_2\in\{-1,1\}}a_1a_2\rchi\left(64(x_1-a_1),64(x_2-a_2)\right).$$ For $M\in\mathbb{N}, M\geq 2$ define $$f_M = \frac{1}{\sqrt{M}\log M}\sum_{k=a_{\alpha,M}}^{b_{\alpha,M}}\frac{1}{k^{2\alpha}}\eta(2^k\cdot),$$ where $$(a_{\alpha,M},b_{\alpha,M})=\begin{cases}
\left(    \left\lfloor (2^M)^{\frac{1}{1-2\alpha}} \right\rfloor,\left\lceil (2^M+M)^{\frac{1}{1-2\alpha}}\right\rceil\right)\quad& 0<\alpha<\frac{1}{2},
\\ \left(M,2^M+M\right) &\alpha = \frac{1}{2}.

\end{cases}$$
Note that $f_M$ satisfies condition (i) of Theorem \ref{largedeformationexistence}.  To complete the proof, we show that given $\epsilon>0$, $\tau>0$, and $\Gamma>0$, one can choose $M$ sufficiently large so that $f_M$ also satisfies conditions (ii) and (iii).

For the remainder of this section, we will use the notation $A\lesssim_\alpha B$ when $A\lesssim B$ and the implied constant depends on $\alpha$, but not $M$.\\
\\
\noindent {\bf $f_M$ satisfies (iii)} We begin by showing $f_M$ satisfies (iii) for sufficiently large $M$. To do this, we take an approach similar to that used in \cite{bourgainli15} and \cite{kwon21}; specifically, we bound the Lagrangian deformation in the $L^\infty$-norm locally from below at positive times by an integral functional which approximates the initial velocity gradient at the origin. 

Note first that the odd-odd symmetry in $f_M$ gives rise to the same odd-odd symmetry in the solution $\omega$ to \eqref{VorticityEquation} with initial data $f_M$. Denoting by $X^t=(X_1^t,X_2^t)$ the particle trajectory map corresponding to this solution, we see that for all $x_1,x_2\in\mathbb{R}$ and for all $t>0$, 
\begin{equation}\label{invariant}
\begin{split}
&X^t_1(0,x_2) = X^t_2(x_1,0) = 0,\\
&X^t(0) = 0,
\end{split}
\end{equation}
so that the origin and axes are invariant under the flow. It follows that
\begin{equation}\label{X_i stays same sign}
    \text{sgn}(x_i)=\text{sgn}(X_i^t(x)),\qquad i=1,2, t>0, x\in\mathbb{R}^2.
\end{equation}
Moreover, if $u=(u_1,u_2)$ denotes the velocity field corresponding to $\omega$, then (\ref{invariant}) implies that for all $x_1,x_2\in\mathbb{R}$ and for all $t>0$,
\begin{equation}\label{u on axes}u_2(X_1^t(x_1,0),0,t)=0=u_1(0,X_2^t(0,x_2),t).
\end{equation}
Therefore, $\nabla u(0,t) $ is diagonal and satisfies $$\nabla u(0,t)=\begin{bmatrix}
    \left(-\partial^2_{x_1,x_2}(-\Delta)^{-1}\omega(t)\right)(0) &0\\ 0& \left(\partial^2_{x_1,x_2}(-\Delta)^{-1}\omega(t)\right)(0)
\end{bmatrix},$$ and
$$\nabla X^t(0) = \exp\left\{\int_0^t\begin{bmatrix}
    \left(-\partial^2_{x_1,x_2}(-\Delta)^{-1}\omega(s)\right)(0) &0\\ 0& \left(\partial^2_{x_1,x_2}(-\Delta)^{-1}\omega(s)\right)(0)
\end{bmatrix}\,ds\right\}.$$
Next, $X^t\circ X^{-t} = Id$ implies $\nabla X^t(X^{-t}(x))=(\nabla X^{-t})^{-1}(x)$ for all $x\in\mathbb{R}^2$. Thus, using $\det\nabla X^{-t} =1$, we have $$\|\nabla X^{-t}(\cdot)\|_{L^\infty(X^t(supp(f_M)))}=\|\nabla X^t(X^{-t}(\cdot))\|_{L^\infty(X^t(supp(f_M)))} = \|\nabla X^t(\cdot)\|_{L^\infty(supp(f_M)).}$$
 So, if $x\in supp(f_M)\subset B_R(0)$ with $x_1,x_2\geq 0$ and $t>0$, using \eqref{X_i stays same sign} we have 
\begin{equation}\label{bounds on x in supp fM}
\begin{split}
\frac{1}{\|\nabla X^{t}\|_{L^\infty(B_R(0))}}X_1^t(x)\leq x_1 \leq X_1^t(x)\|\nabla X^{t}\|_{L^\infty(B_{R}(0))},
\\ \frac{1}{\|\nabla X^{t}\|_{L^\infty(B_R(0))}}X_2^t(x)\leq x_2\leq X_2^t(x)\|\nabla X^t\|_{L^\infty(B_R(0
)).}
\end{split}
\end{equation}
Now we can estimate $|\nabla u (0,t)|$ from below by 
\begin{equation}\label{nabla u(0) lower bound}
    \begin{split}
        &|\nabla u(0,t)|_\infty=\frac{1}{2\pi}\int_{\mathbb{R}^2}\omega(x,t)\partial_{x_1}\frac{-x_2}{|x|^2}\,dx
        = \frac{1}{\pi}\int_{\mathbb{R}^2}\omega(x,t)\frac{x_1x_2}{|x|^4}\,dx
        \\ & \qquad= \frac{4}{\pi}\int_{x_1,x_2>0}f_M(x)\frac{X_1^t(x)X_2^t(x)}{|X^t(x)|^4}\,dx
        \\& \qquad \geq \frac{4}{\pi}\|\nabla X^t\|_{L^\infty(B_R(0))}^{-4}\int_{x_1,x_2>0}f_M(x)\frac{x_1x_2}{|x|^4}\,dx,
    \end{split}
\end{equation}
where we used \eqref{bounds on x in supp fM} to obtain 
\begin{equation*}
\begin{split}
&\frac{X_1^t(x)X_2^t(x)}{|X^t(x)|^4} = \frac{|X^t(x)|^{-2}}{\frac{X_1^t(x)}{X_2^t(x)}+\frac{X^t_2(x)}{X_1^t(x)}}\geq \frac{|X^t(x)|^{-2}}{ \left( \frac{x_1}{x_2} + \frac{x_2}{x_1}\right) \|\nabla X^t \|^2_{L^\infty(B_R(0))}} \geq \|\nabla X^t\|_{L^\infty(B_R(0))}^{-4}\frac{x_1x_2}{|x|^4},
\end{split}
\end{equation*}
for $x=(x_1,x_2)\in supp (f_M)$ with $x_1,x_2>0$ and $t\geq0.$  Denoting the integral $$I_M:=\int_{x_1,x_2>0}f_M(x)\frac{x_1x_2}{|x|^4}\,dx,$$
we have by \eqref{nabla u(0) lower bound}
$$\|\nabla X^t\|_{L^\infty(B_R(0))}\geq |\nabla X^t(0)|\geq\exp\left\{\frac{4}{\pi}I_M\int_0^t\|\nabla X^{s}\|_{L^\infty(B_R(0))}^{-4}\,ds\right\}.$$
Raising both sides to the fourth power and rewriting gives 
\begin{equation*}
\begin{split}
&\qquad \frac{d}{dt}\exp\left\{\frac{16}{\pi}I_M\int_0^t\|\nabla X^{s}\|^{-4}_{L^\infty(B_R(0))}\,ds\right\}\\
& = \frac{16}{\pi}I_M\|\nabla X^{t}\|^{-4}_{L^\infty(B_R(0))}\exp\left\{\frac{16}{\pi}I_M\int_0^t\|\nabla X^{s}\|^{-4}_{L^\infty(B_R(0))}\,ds\right\}
\leq \frac{16}{\pi}I_M.
\end{split}
\end{equation*}
Thus, $$\exp\left\{\frac{16}{\pi}I_M\int_0^t\|\nabla X^{s}\|^{-4}_{L^\infty(B_R(0))}\,ds\right\}\leq 1+\frac{16}{\pi }I_Mt,$$ which is equivalent to $$\int_0^t\|\nabla X^{s}\|^{-4}_{L^\infty(B_R(0))}\,ds\leq \frac{\pi\log\left(1+\frac{16}{\pi}I_Mt\right)}{16I_M \log e}.$$ This implies the existence of some constant $C>0$ independent of $M$ such that
\begin{equation}\label{nabla Xt lower bound}\max_{0\leq s\leq t}\|\nabla X^s\|_{L^\infty(B_R(0))}^4\geq \frac{CI_Mt}{\log(1+CI_Mt)},
\end{equation}
for any $t>0$.

We next estimate $I_M$. Let $$I:=\int_{x_1,x_2>0}\eta(x)\frac{x_1x_2}{|x|^4}\, dx>0.$$
Then, by a change of variables, 
\begin{equation}\label{preprop}
\begin{split}
    I_M=&\int_{x_1,x_2>0}f_M\frac{x_1x_2}{|x|^4}\,dx
=\frac{1}{\sqrt{M}\log M} \sum_{k=a_{\alpha,M}}^{b_{\alpha,M}}\frac{1}{k^{2\alpha}}\int_{x_1,x_2>0}\eta(2^kx)\frac{x_1x_2}{|x|^4}\,dx\\
&\qquad = \frac{1}{\sqrt{M}\log M} \sum_{k=a_{\alpha,M}}^{b_{\alpha,M}}\frac{1}{k^{2\alpha}}\int_{x_1,x_2>0}\eta(x)\frac{x_1x_2}{|x|^4}\,dx = \frac{I}{\sqrt{M}\log M} \sum_{k=a_{\alpha,M}}^{b_{\alpha,M}}\frac{1}{k^{2\alpha}}.
\end{split}
\end{equation}
Since $1/k^{2\alpha}$ is decreasing in $k$, we can write
\begin{equation}\label{sum lower bound}
\begin{split}
   \sum_{k=a_{\alpha,M}}^{b_{\alpha,M}}\frac{1}{k^{2\alpha}} &\geq  \int_{a_{\alpha,M}}^{b_{\alpha,M}}\frac{1}{k^{2\alpha}}\,dk\\
    &=\begin{cases}
   \frac{1}{1-2\alpha}\left(\left\lceil (2^M+M)^\frac{1}{1-2\alpha}\right\rceil^{1-2\alpha}-\left\lfloor 2^{\frac{M}{1-2\alpha}}\right\rfloor^{1-2\alpha} \right) \quad&
   0<\alpha<\frac{1}{2},
   \\ \log(2^M+M)-\log(M) &\alpha=\frac{1}{2},
   \end{cases}
   \\ \qquad &\geq C_\alpha M
\end{split}
\end{equation}
for some constant $C_\alpha>0$ depending on $\alpha$, but not $M$. Substituting this lower bound into (\ref{preprop}) gives
$$I_M=\int_{x_1,x_2>0}f_M\frac{x_1x_2}{|x|^4}\,dx \geq C_\alpha\frac{\sqrt{M}}{\log M}.$$
Thus, substituting this bound for $I_M$ into \eqref{nabla Xt lower bound} with $t=\frac{1}{\log M}$, we obtain
\begin{equation*}\label{largelagrangianM}
\begin{split}
&\max_{0\leq t\leq \frac{1}{\log M}}\|\nabla X^{t}\|_{L^{\infty}(B_R(0))} \geq \left(\frac{CC_\alpha\sqrt{M}}{\log^2 M\log\left(1+\frac{CC_\alpha\sqrt{M}}{\log^2 M}\right)}\right)^{1/4}\\
&\qquad\qquad = \log M\left( \frac{CC_{\alpha}\sqrt{M}}{\log^6 M \log \left(1+ \frac{CC_{\alpha}\sqrt{M} }{\log^2 M} \right)} \right)^{1/4}\gtrsim_\alpha\log M.
\end{split}
\end{equation*}
Finally choosing $M$ sufficiently large yields (iii).\\
\\
{\bf $f_M$ satisfies (ii)}  We now show that for sufficiently large $M$, $f_M$ satisfies (ii).  First note that since $a_{\alpha,M} \geq M$ for each $\alpha$, by a change of variables, $\|f_M\|_{L^1}$ satisfies
\begin{align}\label{fMsmallL1}
\begin{split}
    \|f_M\|_{L^1}&\leq \frac{1}{\sqrt{M}\log M}\sum_{k=a_{\alpha,M}}^{b_{\alpha,M}}\frac{1}{k^{2\alpha}}2^{-2k}\|\eta\|_{L^1} \leq \frac{2^{-M}\|\eta\|_{L^1}}{\sqrt{M}\log M}\sum_{k=a_{\alpha,M}}^{b_{\alpha,M}} 2^{-(k -M)} \lesssim  2^{-M}.
    \end{split}
\end{align}
Moreover, since $\eta(2^kx)\eta(2^jx)=0$ whenever $j\not=k,$
\begin{align}\label{fMsmallLp}
\begin{split}
  %  \|f_M\|_{L^1}&\leq \frac{1}{\sqrt{M}\log M}\sum_{k=a_{\alpha,M}}^{b_{\alpha,M}}\frac{1}{k^{2\alpha}}2^{-2k}\|\eta\|_{L^1} \leq \frac{2^{-M}\|\eta\|_{L^1}}{\sqrt{M}\log M}\sum_{k=a_{\alpha,M}}^{b_{\alpha,M}} 2^{-(2k -M)} \lesssim  2^{-M},
 \|f_M\|_\infty&\leq \frac{1}{\sqrt{M}\log M},
    \end{split}
    \end{align}
so that $\|f_M\|_{\infty}$ can be made arbitrarily small for sufficiently large choice of $M$.  

To complete the proof that $f_M$ satisfies (ii), it remains to estimate $\| f_M\|_{H^{1,\alpha}}$.  We first use interpolation and the above estimates to write
\begin{equation}\label{fMsmallL2}
\begin{split}
&\| f_M \|_{L^2} \leq \| f_M \|^{1/2}_{\infty}\| f_M \|^{1/2}_{L^1} \lesssim 2^{-M/2}. 
\end{split}
\end{equation}
We once again use that $\eta(2^kx)\eta(2^jx)=0$ whenever $j\not=k$, and that 
\begin{equation}\label{sumtointegral}
\sum_{k=a_{\alpha,M}}^{b_{\alpha,M}}\frac{1}{k^{4\alpha}}<\sum_{k=a_{\alpha,M}}^{b_{\alpha,M}}\frac{1}{k^{2\alpha}}<\int_{a_{\alpha,M}-1}^{b_{\alpha,M}}\frac{1}{k^{2\alpha}}\,dk \lesssim_\alpha M,
\end{equation}
where the integral is computed in (\ref{sum lower bound}), to write
    \begin{align}\label{fMsmall}
        \|\nabla f_M \|_{L^2}  = \frac{1}{\sqrt{M}\log M}\left(\sum_{k=a_{\alpha,M}}^{b_{\alpha,M}}\frac{1}{k^{4\alpha}}\|\nabla \eta\|_{L^2}^2\right)^{\frac{1}{2}}\lesssim_\alpha \frac{1}{\log M}.
\end{align}
For the $H^{1,\alpha}$-seminorm of $f_M$, we have by (\ref{real variable definition for s=1}), 
\begin{align}\label{H1alphanormoffM}
   \notag M\log^2M|f_M|_{H^{1,\alpha}}^2 &\lesssim_\alpha \int_{\mathbb{R}^2}\int_{|x-y| \leq\frac{1}{2}}\frac{\left|\sum_{k=a_{\alpha,M}}^{b_{\alpha,M}}\frac{2^{k}}{k^{2\alpha}}\left(\nabla\eta(2^kx)-\nabla\eta(2^ky)\right)\right|^2}{|x-y|^2\log^{1-2\alpha}(|x-y|^{-1})}\, dy\, dx
   \\ & \lesssim_\alpha \sum_{k=a_{\alpha,M}}^{b_{\alpha,M}}\left( L_k+C_k \right),
\end{align} where
\begin{equation*}
    L_k:=\int_{\mathbb{R}^2}\int_{|x-y|\leq \frac{1}{2}}\frac{\frac{2^{2k}}{k^{4\alpha}}\left|\nabla\eta(2^k x)-\nabla\eta(2^ky)\right|^2}{|x-y|^2\log^{1-2\alpha}(|x-y|^{-1})}\,dy\,dx
\end{equation*}
and 
\begin{equation}\label{Ckdef}
    C_k:=2\sum_{j=a_{\alpha,M}}^k\frac{2^{j+k}}{(kj)^{2\alpha}}\int_{\mathbb{R}^2}\int_{|x-y|\leq \frac{1}{2}} \frac{\left|\nabla\eta(2^kx)-\nabla\eta(2^ky)\right|\left|\nabla\eta(2^jx)-\nabla\eta(2^jy)\right|}{|x-y|^2\log^{1-2\alpha}(|x-y|^{-1})}\, dy\, dx.
\end{equation}
To estimate $L_k$, we utilize the Fourier definition of the $H^{0,\alpha}$-seminorm and apply a change of variables to write
\begin{equation*}
\begin{split}
&L_k \lesssim_\alpha\frac{2^{2k}}{k^{4\alpha}}|\nabla \eta(2^k\cdot)|_{H^{0,\alpha}}^2\simeq_\alpha \frac{2^{2k}}{k^{4\alpha}}\int_{\mathbb{R}^2}\log^{2\alpha}(|\xi|+1)2^{-4k}|\widehat{\nabla\eta}(2^{-k}\xi)|^2 \, d\xi \\ &\qquad\qquad\simeq_\alpha \frac{1}{k^{4\alpha}}\int_{\mathbb{R}^2}\log^{2\alpha}(2^k|\xi|+1)|\widehat{\nabla\eta}(\xi)|^2 \, d\xi
    \\ &\qquad\qquad\lesssim_\alpha  \frac{1}{k^{4\alpha}}\left(\int_{\mathbb{R}^2}\left(\log^{2\alpha}(2^k)+\log^{2\alpha}(|\xi|+1)\right)|\widehat{\nabla\eta}(\xi)|^2 \, d\xi\right)
    \\  &\qquad\qquad\lesssim_\alpha  \frac{1}{k^{2\alpha}}\left(\|\nabla \eta\|_{L^2}^2+|\nabla\eta|_{H^{0,\alpha}}^2\right).
\end{split}
\end{equation*}
Thus, an application of (\ref{sumtointegral}) gives 
\begin{equation}\label{Halphanormliketerms}
    \sum_{k=a_{\alpha,M}}^{b_{\alpha,M}}L_k \lesssim_\alpha   M\|\nabla\eta\|_{H^{0,\alpha}}^2.
\end{equation}

To estimate the cross terms $C_k$, we utilize the the support of $\nabla\eta$ to manipulate the integrand.  First note that $j<k$ implies $\nabla\eta(2^kx)\nabla\eta(2^jx)=0$ for all $x\in \mathbb{R}^2$. Also, the support of $\nabla\eta$ ensures that $\nabla\eta(2^kx)\nabla\eta(2^jy)\not=0$ only if
\begin{equation*}
\begin{split}
&|x|<2^{-k}(\sqrt{2}+2^{-6}),\\
&|y|>2^{-j}(\sqrt{2}-2^{-6}).
\end{split}
\end{equation*}
Thus by the triangle inequality,
\begin{equation*}
\begin{split}
&|2^{-j}y-2^{-k}x|\geq 2^{-2j}(\sqrt{2}-2^{-6})-2^{-2k}(\sqrt{2}+2^{-6})\\
&\qquad
>2^{-2j+1/4}-2^{-2k+3/4}  = 2^{-2j-1/2}(2^{3/4} - 2^{-2(k-j) + 5/4}) %\geq 2^{-2j-1/2}(2^{3/4} - 2^{-2 + 5/4}) 
\geq 2^{-2j-1/2},
\end{split}
\end{equation*}
where we used that $k-j\geq 1$.
%where we've used the fact that $b-a\geq 1$ implies $2^{-a}-2^{-b}\geq 2^{-a-1}.$  
Finally, by the same reasoning, if $j<k$ and $\nabla\eta(2^jx)\nabla(2^ky)\not=0$, interchanging the roles of $x$ and $y$ gives $|2^{-j}x-2^{-k}y|\geq 2^{-2j-1/2}.$ Thus, for $j<k$, a change of variables in $x$ and $y$ followed by a change of variables in $|x-y|$ gives
\begin{align*}
    &\qquad\int_{\mathbb{R}^2}\int_{|x-y|\leq \frac{1}{2}} \frac{\left|\nabla\eta(2^kx)-\nabla\eta(2^ky)\right|\left|\nabla\eta(2^jx)-\nabla\eta(2^jy)\right|}{|x-y|^2\log^{1-2\alpha}(|x-y|^{-1
    })}\, dy\, dx
    \\ &\leq \int_{\mathbb{R}^2}\int_{|x-y|\leq \frac{1}{2}} \frac{\left|\nabla\eta(2^kx)\nabla\eta(2^jy)\right|}{|x-y|^2\log^{1-2\alpha}(|x-y|^{-1})} \, dy\, dx+\int_{\mathbb{R}^2}\int_{|x-y|\leq \frac{1}{2}} \frac{\left|\nabla\eta(2^jx)\nabla\eta(2^ky)\right|}{|x-y|^2\log^{1-2\alpha}(|x-y|^{-1})} \, dy\, dx
    \\ &\qquad\leq2^{-2k-2j+1}\|\nabla\eta\|_\infty^2\int_{\text{supp}(\eta)} \left(\int_{2^{-2j-1/2}}^{\frac{1}{2}}\frac{2\pi}{h\log^{1-2\alpha}(h^{-1})}dh\right) \, dx
    \\ &\qquad =2^{-8-2k-2j}\pi^2\|\nabla\eta\|_\infty^2\left(\left(2j+\frac{1}{2}\right)^{2\alpha}-1\right)\lesssim_\alpha 2^{-2k-2j}j^{2\alpha}\|\nabla\eta\|_\infty^2.
\end{align*}
Substituting the above estimate into (\ref{Ckdef}) and summing over $k$, we conclude that
\begin{equation}\label{H1alphanormfMcrossterms}
\begin{split}
&\sum_{k=a_{\alpha,M}}^{b_{\alpha,M}}C_k \lesssim_\alpha\|\nabla\eta\|_\infty^2\sum_{k=a_{\alpha,M}}^{b_{\alpha,M}}\sum_{j=a_{\alpha,M}}^k\frac{1}{2^{k+j}k^{2\alpha}}\\
&\qquad \lesssim_\alpha\|\nabla\eta\|_\infty^2\sum_{k=a_{\alpha,M}}^{b_{\alpha,M}}\sum_{j=a_{\alpha,M}}^k\frac{1}{2^{k+j}} \lesssim_\alpha 2^{-2a_{\alpha,M}}\|\nabla\eta\|_\infty^2\lesssim_\alpha 4^{-M}\|\nabla\eta\|_\infty^2,
\end{split}
\end{equation}
where we again used that $a_{\alpha,M} \geq M$ for all $\alpha$ to obtain the last inequality.  Combining (\ref{H1alphanormoffM}), (\ref{Halphanormliketerms}), and (\ref{H1alphanormfMcrossterms}) gives 
\begin{equation}\label{H1alphanormfM}
    |\nabla f_M|_{H^{0,\alpha}}\lesssim_\alpha \frac{\|\eta\|_{H^{1,\alpha}}+\|\nabla\eta\|_\infty}{\log M}. 
\end{equation}
Finally, (\ref{fMsmallL2}), (\ref{fMsmall}), and (\ref{H1alphanormfM}) together imply that $\| f_M \|_{H^{1,\alpha}}$ can be made small for sufficiently large choice of $M$.  This completes the proof of Theorem \ref{largedeformationexistence}.   
\end{proof}
As a corollary of Theorems \ref{localnorminflationtheorem} and \ref{largedeformationexistence}, we have the existence of a solution to $(\ref{VorticityEquation})$ experiencing arbitrarily large growth in $H^{1,\alpha}$-norm.
\begin{cor}\label{local corollary}
    Let $\varepsilon>0$, $A>0$, and $\tau>0$ be arbitrary. Then there exists a function $\omega_0^{\varepsilon,\tau}\in C^\infty_c(\mathbb{R}^2)$ satisfying:\\
    \\ (i) $\|\omega_0^{\varepsilon,\tau}\|_\infty+\|\omega_0^{\varepsilon,\tau}\|_{H^{1,\alpha}}<\varepsilon$,
    \\ (ii) supp$(\omega_0^{\varepsilon,\tau})\subset B_1(0),$
    \\ (iii) The solution $\omega^{\varepsilon,\tau}$ of (\ref{VorticityEquation}) with initial data $\omega_0^{\varepsilon,\tau}$ satisfies $$\max_{0\leq t\leq \tau}\|\omega^{\varepsilon,\tau}(t)\|_{H^{1,\alpha}}>A.$$
\end{cor}
\begin{proof}
      Let $c_{\alpha}$ be as in Theorem \ref{localnorminflationtheorem}.  It follows from Theorem \ref{largedeformationexistence} with $\Gamma=\max \{2^{13}, c^2_{\alpha}/{\varepsilon}^2, A^3 \}$ that there exists $f=\omega_0$  satisfying conditions (i) and (ii) of Theorem \ref{localnorminflationtheorem}.  Thus, by Theorem \ref{localnorminflationtheorem},  there exists $\beta$ satisfying (\ref{smallperturb}), and, in particular, $$|\beta|_{H^{1,\alpha}} \leq \frac{c_{\alpha}}{\sqrt{L}} \leq \frac{c_{\alpha}}{\sqrt{\Gamma}} \leq \varepsilon.$$  Set $\omega_0^{\varepsilon,\tau} = f + \beta$.  Then $\omega_0^{\varepsilon,\tau}$ satisfies conditions (i) and (ii) of Corollary \ref{local corollary}; specifically, %Furthermore, by (\ref{smallperturb}), (\ref{largelagrangianM}), (\ref{fMsmall}), (\ref{fMsmallLp}) and (\ref{H1alphanormfM}), the perturbed initial data $\widetilde{f_M}$, constructed in Theorem \ref{localnorminflationtheorem}, satisfies 
    $$\|\omega_0^{\varepsilon,\tau}\|_\infty +\|\omega_0^{\varepsilon,\tau}\|_{H^{1,\alpha}}< 4\varepsilon$$ and supp$(\omega_0^{\varepsilon,\tau}) \subset B_1(0)$.  Moreover, it follows from Theorem \ref{localnorminflationtheorem} and the definition of $\Gamma$ that $\omega_0^{\varepsilon,\tau}$ satisfies condition (iii). %,e  $\omega_0^{\varepsilon,\tau}=\widetilde{f_M}$ satisfies (i) and (ii). If need be, $M$ can be taken larger to guarantee (iii) follows from the conclusion (\ref{largeH1alphanormconclusion}). 
\end{proof}
%%%%%%%%%%%%%%%%%%%%%%%%%%%%%%%%%
%%%%%%%%%%%%%%%%%%%%%%%%%%%%%%%%%
%%%%%%%%%%%%%%%%%%%%%%%%%%%%%%%%%
\section{Interaction of Vorticity Patches and Proof of Theorem \ref{main}}\label{Section Proof of main theorem}
%%%%%%%%%%%%%%%%%%%%%%%%%%%%%%%%%
%%%%%%%%%%%%%%%%%%%%%%%%%%%%%%%%%
%%%%%%%%%%%%%%%%%%%%%%%%%%%%%%%%%
In this section, we prove Theorem \ref{main}.  Our first step is to use Corollary \ref{local corollary} to construct an infinite collection of vorticity patches, each with small $H^{1,\alpha}$-norm.  We construct this collection in such a way that the infinite sum of these patches yields a non-compactly supported initial vorticity $\omega_0\in H^{1,\alpha}(\mathbb{R}^2)$ whose corresponding solution exhibits blow-up in $H^{1,\alpha}$-norm at positive times. In order to control the interaction of these patches, we rely on the following proposition.

\begin{prop}[\cite{bourgainli15}, Proposition 5.3]\label{BL interaction prop} Let $\{\omega_0^j\}_{j=1}^\infty$ satisfy $\omega_0^j \in C^\infty_c(\mathbb{R}^2)$ and supp$(\omega_0^j)\subset B_1(0)$ for each $j\in\mathbb{N}$, and let $\{\omega_0^j\}_{j=1}^\infty$ be such that
\begin{equation}\label{sum of omegaj less than c1}
    \sum_{j=1}^\infty\left(\|\omega_0^j\|_{H^1}^2+\|\omega_0^j\|_{L^1}\right)+\sup_{j\geq 1}\|\omega_0^j\|_\infty<1.
\end{equation}
Then there exists a sequence of points $\{x_j\}$ in $\mathbb{R}^2$ such that the following hold:\\
\\(1) For $k\not=j$, $$B_{64}(x_j)\cap B_{64}(x_k)=\emptyset.$$
\\(2) If we define $$\omega_0(x)=\sum_{j=1}^\infty\omega_0^j(x-x_j),$$
then $\omega_0\in L^1 \cap C^\infty(\mathbb{R}^2)$.
\\ (3) If $\omega$ is the solution to $$\begin{cases}
    \partial_t\omega+\nabla^\perp(-\Delta)^{-1}\omega\cdot\nabla\omega=0 \quad &\text{on }\mathbb{R}^2\times (0,\infty),
    \\ \omega|_{t=0}=\omega_0 \quad &\text{on }\mathbb{R}^2,
\end{cases}$$ 
then for all $t\in[0,1]$, $\omega(\cdot,t)\in L^1\cap C^\infty(\mathbb{R}^2)$ and 
$$\text{supp}\left(\omega(\cdot,t)\right)\subset\bigcup_{j=1}^\infty B_4(x_j).$$
\\ (4) If we denote by $\omega_j$ the solution to (\ref{VorticityEquation}) with initial data $\omega_0^j(\cdot-x_j)$, i.e. $$\begin{cases}
    \partial_t\omega_j+\nabla^\perp(-\Delta)^{-1}\omega_j\cdot\nabla\omega_j=0 \quad &\text{on }\mathbb{R}^2\times (0,\infty),
    \\ \omega_j(x,0)=\omega_0^j(x-x_j) \quad &\text{on }\mathbb{R}^2,
\end{cases}$$ 
then $$\lim_{j\to \infty}\max_{0\leq t\leq 1}\|\omega_j(\cdot,t)-\omega(\cdot,t)\|_{H^2(B_4(x_j))}=0.$$
    
\end{prop}

\begin{proof}[Proof of Theorem \ref{main}]
We first apply Corollary \ref{local corollary} to obtain a sequence of functions $\{\omega_0^j\}_{j=1}^\infty$ such that for every $j\in\mathbb{N}$, $\omega^j_0 \in C^{\infty}_c\left(B_1(0)\right)$ and the following hold: 
\begin{equation}\label{proof of main theorem}
    \begin{split}
        %&\|\omega_0^j\|_\infty\leq \frac{1}{4}
        %\\
&\|\omega_0^j\|_{L^\infty}+\|\omega_0^j\|_{L^1}+\|\omega_0^j\|_{H^{1,\alpha}}^2<2^{-(j+1)},
        \\ &\max_{0\leq t\leq \frac{1}{j}}\|\omega_j\|_{H^{1,\alpha}}>j,
    \end{split}
\end{equation}
where $\omega_j$ solves $(\ref{VorticityEquation}$) with initial data $\omega_0^j$. Thus, $\{ \omega_0^j\}$ satisfy the hypotheses of Proposition \ref{BL interaction prop}, which gives us a sequence $\{x_j\}_{j=1}^\infty$ in $\mathbb{R}^2$ such that the solution $\omega$ to $$\begin{cases}
    \partial_t\omega+u\cdot\nabla\omega=0 \quad &\text{on }\mathbb{R}^2\times (0,\infty),
    \\ u= \nabla^\perp(-\Delta)^{-1}\omega\quad &\text{on }\mathbb{R}^2\times[0,\infty),
    \\ \omega|_{t=0}=\sum_{j=1}^\infty\omega_0^j(\cdot-x_j) \quad &\text{on }\mathbb{R}^2,
\end{cases}$$ 
satisfies $$\lim_{j\to\infty}\max_{0\leq t\leq \frac{1}{j}}\|\omega(\cdot,t)-\omega_j(\cdot,t)\|_{H^2(B_{4}(x_j))}=0.$$ Because $\|f\|_{H^{1,\alpha}}\lesssim_\alpha\|f\|_{H^2}$ for all $f\in H^{2}(\mathbb{R}^2)$, it follows from (\ref{proof of main theorem}) that $$\limsup_{t\to0^+}\|\omega(\cdot,t)\|_{H^{1,\alpha}}=\infty.$$ Clearly, this gives $\limsup_{t\to0^+}\|u(\cdot,t)\|_{H^{2,\alpha}}=\infty$, which concludes the proof.
\end{proof}

\section*{Acknowledgments} \noindent EC is grateful to the Simons Foundation
for support through Grant 429578.

\bibliographystyle{alpha}
\bibliography{illposedbibliography}

\begin{thebibliography}{CNSS17}

\bibitem[BBC16]{bohunbouchutcrippa16}
Anna Bohun, Fran{\c{c}}ois Bouchut, and Gianluca Crippa.
\newblock Lagrangian solutions to the {2D} {Euler} system with {$L^1$} vorticity and infinite energy.
\newblock {\em Nonlinear Anal.}, 132:160--172, 2016.

\bibitem[BCD11]{bahourichemindanchin11}
Hajer Bahouri, Jean-Yves Chemin, and Rapha\"el Danchin.
\newblock {\em Fourier {A}nalysis and {N}onlinear {P}artial {D}ifferential {E}quations}.
\newblock Springer, 2011.

\bibitem[BKM84]{bealekatomajda84}
J~Thomas Beale, Tosio Kato, and Andrew Majda.
\newblock Remarks on the breakdown of smooth solutions for the {3-D} {Euler} equations.
\newblock {\em Comm. Math. Phys.}, 94(1):61--66, 1984.

\bibitem[BL15]{bourgainli15}
Jean Bourgain and Dong Li.
\newblock Strong ill-posedness of the incompressible {Euler} equation in borderline {S}obolev spaces.
\newblock {\em Inventiones mathematicae}, 201:97--157, 2015.

\bibitem[BL21]{bourgainli21}
Jean Bourgain and Dong Li.
\newblock Strong ill-posedness of the {3D} incompressible {Euler} equation in borderline spaces.
\newblock {\em International Mathematics Research Notices}, 2021(16):12155--12264, 2021.

\bibitem[BN19]{bruenguyen19}
Elia Bru{\'e} and Quoc-Hung Nguyen.
\newblock On the {Sobolev} space of functions with derivative of logarithmic order.
\newblock {\em Advances in Nonlinear Analysis}, 9(1):836--849, 2019.

\bibitem[Cha02]{chae02}
Dongho Chae.
\newblock On the well-posedness of the {Euler} equations in the {Triebel-Lizorkin} spaces.
\newblock {\em Communications on Pure and Applied Mathematics: A Journal Issued by the Courant Institute of Mathematical Sciences}, 55(5):654--678, 2002.

\bibitem[Cha03]{chae03}
Dongho Chae.
\newblock On the {Euler} equations in the critical {Triebel-Lizorkin} spaces.
\newblock {\em Archive for rational mechanics and analysis}, 170(3):185--210, 2003.

\bibitem[Cha04]{chae04}
Dongho Chae.
\newblock Local existence and blow-up criterion for the {Euler} equations in the {Besov} spaces.
\newblock {\em Asymptotic Analysis}, 38(3-4):339--358, 2004.

\bibitem[CNSS17]{crippaseisspirito17}
Gianluca Crippa, Camilla Nobili, Christian Seis, and Stefano Spirito.
\newblock Eulerian and {Lagrangian} solutions to the continuity and {Euler} equations with {$L^{1}$} vorticity.
\newblock {\em SIAM J. Math. Anal.}, 49(5):3973--3998, 2017.

\bibitem[EJ17]{elgindijeong17}
Tarek~Mohamed Elgindi and In-Jee Jeong.
\newblock Ill-posedness for the incompressible {Euler} equations in critical {Sobolev} spaces.
\newblock {\em Annals of PDE}, 3:1--19, 2017.

\bibitem[HR25]{harrisonradke25}
Nicholas Harrison and Zachary Radke.
\newblock Well-posedness for the {Euler} equations in function spaces of generalized smoothness.
\newblock {\em arXiv preprint https://arxiv.org/abs/2510.02626}, 2025.

\bibitem[JK21]{jeongkim21}
In-Jee Jeong and Junha Kim.
\newblock A simple proof of ill-posedness for incompressible {Euler} equations in critical {Sobolev} spaces.
\newblock {\em arXiv preprint arXiv:2111.14078}, 2021.

\bibitem[KP86]{katoponce86}
Tosio Kato and Gustavo Ponce.
\newblock Well-posedness of the {Euler} and {Navier-Stokes} equations in the {Lebesgue} spaces {$L^p_s(\mathbb{R}^2)$}.
\newblock {\em Rev. Mat. Iberoam.}, 2:73--88, 1986.

\bibitem[Kwo21]{kwon21}
Hyunju Kwon.
\newblock Strong ill-posedness of logarithmically regularized {2D} {Euler} equations in the borderline {Sobolev} space.
\newblock {\em Journal of Functional Analysis}, 280(7):108822, 2021.

\bibitem[PH23]{hwangpak23}
Hee~Chul Pak and Jun~Seok Hwang.
\newblock Persistence of the solution to the {Euler} equations in the end-point critical {Triebel-Lizorkin} space ${F}^{d+1}_{1,\infty}(\mathbb{R}^d)$.
\newblock {\em arXiv preprint arXiv:2302.13295}, 2023.

\bibitem[SM93]{stein93}
Elias~M Stein and Timothy~S Murphy.
\newblock {\em Harmonic analysis: real-variable methods, orthogonality, and oscillatory integrals}, volume~3.
\newblock Princeton University Press, 1993.

\bibitem[Vis98]{vishik98}
Misha Vishik.
\newblock Hydrodynamics in {Besov} spaces.
\newblock {\em Archive for Rational Mechanics and Analysis}, 145(3):197--214, 1998.

\bibitem[Wol33]{wolibner33}
Witold Wolibner.
\newblock Un theor{\`e}me sur l'existence du mouvement plan d'un fluide parfait, homog{\`e}ne, incompressible, pendant un temps infiniment long.
\newblock {\em Math. Z.}, 37:698--726, 1933.

\bibitem[Yud63]{yudovich63}
Victor~I Yudovich.
\newblock Non-stationary flow of an ideal incompressible liquid.
\newblock {\em U.S.S.R. Comput. Math. and Math. Phys.}, 3(6):1407--1456, 1963.

\end{thebibliography}

\end{document}